\documentclass[11pt]{amsart}
\usepackage{amsmath,amssymb,amsbsy,amsfonts,amsthm,latexsym,amsopn,amstext,amsxtra,epic, euscript,amscd,indentfirst, verbatim}
\usepackage{hyperref}
\usepackage{graphicx}
\usepackage[margin=1.4in]{geometry}

\usepackage{tikz,xcolor}

\usetikzlibrary{arrows,decorations.pathmorphing,backgrounds,fit,automata,calc}
\usetikzlibrary{positioning}

\newtheorem{theorem}{Theorem}
\newtheorem*{theorem*}{Theorem}
\newtheorem*{lemma*}{Lemma}
\newtheorem{conjecture}[theorem]{Conjecture}

\newtheorem{lemma}[theorem]{Lemma}
\newtheorem{cor}[theorem]{Corollary}
\newtheorem*{obs*}{Observation}

\newcommand{\comments}[1]{} 
\newcommand{\appropto}{\mathrel{\vcenter{
  \offinterlineskip\halign{\hfil$##$\cr
    \propto\cr\noalign{\kern2pt}\sim\cr\noalign{\kern-2pt}}}}}

\def\R{\mathbb R}

\def\dt{\frac{d}{dt}}
\def\twodt{\frac{d^2}{dt^2}}

\begin{document}

\title[On Hall's conjecture]{Hall's Conjecture on Extremal Sets \\ for Random Triangles}
 
\author[University of Michigan]{Gabriel Khan} 
 
\email{gabekhan@umich.edu}
 
\date{\today}
 
\maketitle

\begin{abstract}
In this paper we partially resolve Hall's conjecture~\cite{GH} on random triangles.  We consider the probability that three points chosen uniformly at random from a bounded convex region of the plane form an acute triangle. Hall's conjecture is states that this probability is maximized by the disk. This can be interpreted as a probabilistic version of the isoperimetric inequality.
 
We first prove that the disk is a weak local maximum among bounded domains in $\R^2$ and that the ball is a weak local maximum in $\R^3$. In $\R^2$, we then prove a local $C^{2,\frac{1}{2}}$-type estimate on the probability in the Hausdorff topology. This enables us to prove that the disk is a strong local maximum in the Gromov-Hausdorff topology (modulo congruences).
Finally, we give an explicit upper bound on the isoperimetric ratio for a region which maximizes the probability and show how this reduces verifying the full conjecture to a finite, though currently intractable, calculation.
\end{abstract}



\section{What is the probability that a randomly chosen triangle is acute? }

A classic problem in geometric probability, dating back to the 19th century, is determining the shape of a random triangle. This problem was popularized by Charles Dodgson (better known as Lewis Carroll) in his lesser known book  \emph{Pillow Problems}~\cite{LC}, when he asked for the probability that a random triangle is acute. 

After some initial thought, it becomes clear that the answer depends on the definition of ``random" triangles. In his book, Carroll assumed that the longest side of the triangle is a known line segment, and that the third vertex is chosen uniformly at random given that segment. Others have proposed different methods for choosing random triangles; such as randomly choosing the side lengths or angles or in terms or in terms of a measure on the moduli space of triangles (see~\cite{GR},~\cite{ES},~\cite{CNSS}). Stephen Portnoy~\cite{SP} wrote a good overview of the history and explained how many of the natural methods for choosing random triangles have unsettling properties.

For our work, we consider random triangle defined by choosing the three vertices at random. This approach was first considered by W.S.B. Woolhouse~\cite{WW} in 1861. He claimed that the probability of choosing an acute triangle was $(4\pi^{-2} -1/8)$, but did not specify the probability distribution from which he picked the points. Consequently, the problem was ill-posed~\cite{ES}. However if the three vertices are chosen uniformly at random from a disk, then the answer is indeed correct.
Other people have studied the question when the vertices are chosen using different distributions. Explicit calculations have been done in the following cases:
\begin{enumerate}
\item The vertices are chosen from a Gaussian distribution (in which case the probability is exactly $1/4$)~ \cite{KM}.
\item The vertices are chosen uniformly at random in a rectangle~\cite{EL}.
\item The vertices are chosen uniformly at random in a triangle~\cite{VA}.
\end{enumerate}

\subsection{Hall's conjecture}

  In 1982, Glen Hall ~\cite{GH} computed the probability that three points chosen uniformly at random from the $n$-ball form an acute triangle. He did so by applying Baddeley's generalization of Crofton's differential equation \cite{AB} to simplify the relevant integrals.
He also observed that the $n$-ball is a critical point for this probability and that as $n$ increases, the probability of choosing an acute triangle increases, converging to 1 in the limit. In light of this, he conjectured among convex domains in $\mathbb{R}^n$, the probability that three randomly chosen points form an acute triangle is maximized when the domain is the $n$-ball. 
 
Heuristically, Hall's conjecture seems very likely to be true. In order for a domain to maximize the probability, the first variation of the probability must vanish. From Hall's work, it appears that this can only occur when the domain has $SO(n)$-symmetry. Furthermore, increasing the dimension increases the probability, so there is no apparent degeneracy (such as the probability being maximized on a lower dimensional subset).  Nevertheless, the conjecture remains open since translating statements about probability into geometric data is a difficult problem. 
  
 
   \subsection{Our results}
 
 
 Before stating the results precisely, we introduce our notation and conventions on local extrema.
  We will typically denote a convex region in $\R^n$ by $S$. When we are considering a 1-parameter family of such regions, with parameter $t$, we will denote it by $S(t)$. In this context $S(0)$ will denote the $n$-ball centered at the origin. We now consider the triangle $\triangle T$ obtained by choosing three vertices from $S$ i.i.d. uniformly. We will use $p(S)$ to denote the probability that $\triangle T$ is an acute triangle, and we use $F_S(\phi)$ to denote the probability that the largest angle of $\triangle T$ is less than or equal to $\phi$.

For this work, we distinguish between local extrema and weak local extrema. Following the convention of \cite{GF}, we say that the functional $p(S)$ has a weak local maximum for $S = S_0$ if there exists an $\epsilon > 0$ such that $p(S)-p(S_0) < 0$ for all $S$ in the domain of $p$ which satisfy the $d_1( S, S_0)  < \epsilon$, where $d_1$ corresponds to distance in the $C^1$ norm. 
   
   We also consider the extrema in the $C^0$ topology, which we call local extrema. We say that the functional $p(S)$ has a local maximum for $S = S_0$ if there exists an $\epsilon > 0$ such that $p(S)-p(S_0) < 0$ for all $S$ in the domain of $p$ which satisfy $d_{H.}( S, S_0)  < \epsilon$. Here, $d_{H.}$ is the Hausdorff distance, which yields a coarser topology. 

With these preliminaries covered, we can introduce our main results.
 
 \begin{theorem}
Among bounded convex domains in $\R^2$, the disk $D$ is a weak local maximum for $p$.
 \end{theorem}
More precisely, we show that the second variation of the probability is strongly negative at $D$. An immediate consequence is that given a $C^2$ one-parameter family of convex regions $S(t)$ such that $S(0)$ is the disk and whose first variation modulo congruences is non-zero, $p(S(t))$ has a local maximum at 0.  We also prove the corresponding result in $\R^3$. 

    \begin{theorem}
  Among bounded convex domains in $\R^3$, the ball $B$ is a weak local maximum for $p$. 
\end{theorem}

In two dimensions, we are able to show that the same result holds in the stronger sense. 

\begin{theorem}
The disk is a local maximum for $p$ for bounded convex domains in $\R^2$. More precisely, there exists $\epsilon$ such that for any convex region $S$ with $d_{H.}(D,S) < \epsilon$, then $p(S) < p(D)$. 
\end{theorem}

 In Section 6, we prove that regions that are far from the disk have small probability of choosing acute triangles.
 
 \begin{theorem}
Let $S$ be a convex subset of $\R^2$ whose isoperimetric ratio is greater than $\frac{7688}{15}$. Then $p(S) < p(D)$. Furthermore, if we denote the isoperimetric ratio of $S$ by $R$, then $p(S) \lesssim R^{-1}$.
\end{theorem}

Finally, in Section 7, we discuss how, given Theorems 3 and 4, the proof of Hall's original conjecture is reduced to a finite computation. Unfortunately, at this point the computations are infeasible. 
With care, it might be possible to reduce the problem to the point that a computer can check the remaining cases and so prove the full conjecture.

\subsection{A discussion of the proofs}

Our approach is variational, but is directly inspired by Baddeley's generalization of Crofton's differential equation \cite{AB}.  Crofton's differential equation was first derived by Morgan Crofton in 1885. It was originally used as a tool to simplify complicated geometric integrals that appear in geometric probability. However, it is also a prototype for modern variational techniques in geometric probability. 
    
    In Section 2, we prove Theorem 1. To do so, we use calculus of variations to express the second variation of $p(S)$ along some deformation of $D$. In order to compute this variation explicitly, we use the Fourier transform of the deformation of $S$. When we do so, a term appears that integrates the autocorrelation of the deformation against a particular function $A_2(\theta)$. The Fourier coefficients of $A_2(\theta)$ display an interesting pattern that, when combined with the Plancherel theorem and correlation theorem, forces the second variation to be strongly negative.

 Section 3 contains the proof of Theorem 2, which is similar to that of Theorem 1 but more technical. We again decompose the deformation, except we must do so in terms of spherical harmonics. The autocorrelation integral is now defined in terms of elements of $SO(3)$, so we use the Plancherel theorem and correlation theorem for this group to calculate the second variation of the probability. Furthermore, the integrals needed are too involved to calculate by hand and so require Mathematica. It is likely possible to repeat the calculation in higher dimensions, but the calculations quickly become intractable. For the three dimensional case, we have written a Mathematica notebook which is available for download \cite{Website}. Theorems 1 and 2 hold in weaker regularity than $C^2$, and the specific regularity needed for the variation is discussed at the start of the proof of Theorem 1.

In Section 4, we remove the regularity assumptions of Theorem 1 and proves that the disk is a local maximum for the probability in the Hausdorff topology. We consider an arbitrary convex region which is Gromov-Hausdorff close to the disk. We first prove some estimates and use the homotopy invariance of the degree of sphere maps in order to establish the existence of a canonical homotopy from $D$ to $S$. Along this homotopy, we prove that the second derivative of the probability is uniformly H\"older-$1/2$ continuous. This can be thought of as a local $C^{2, 1/2}$ estimate on the probability near the disk.  Interestingly, for general convex regions near the disk, it seems we cannot use this approach to prove a uniform $C^3$ estimate along the canonical homotopy. As such, a $C^{2,\alpha}$ estimate is likely optimal. By Theorem 1, the second derivative of the probability is strictly negative at $0$, and so is strictly negative along the entire homotopy when $S$ is sufficiently close to $D$. Using this, we find a super-solution for the probability to find an upper bound on the probability which is smaller than that of the disk.
    
In Section 5, we prove that if a region if a convex domain is far from a disk in $C^0$ topology, then the probability of choosing an acute triangle is very small. A qualitative version of this result is well-known in the literature. Convex regions with large isoperimetric ratio are long and thin, so the three points are nearly collinear with large probability and thus have very small probability of forming an acute triangle. We make this intuition precise with an explicit estimate, which is helpful for the last section. Apart from providing further evidence for Hall's conjecture, this shows that the supremum of the probability occurs in a compact set of isoperimetric ratios. Since the only non-compact parts of the moduli space of convex figures correspond to blow-ups of the isoperimetric ratio, this shows that the supremum is achieved on a compact subset in the $C^0$ topology. Theoretically, this reduces the proof of the full conjecture to a finite computation, although it is not currently tractable. We discuss how this can be done in Section 6.

     \subsection{Isoprobabilistic inequalities and other applications}

It is instructive to consider Hall's conjecture in the context of general isoperimetric inequalities. The classic isoperimetric inequality states that of all unit volume convex regions, the measure of the boundary is minimized when the convex region is an $n$-ball. The idea has since been generalized greatly and has important implications in geometry, physics, and functional analysis (for a broad overview, see ~\cite{EP}). We propose that Hall's conjectured inequality can be interpreted as a probabilistic version of such an inequality.  

We say that an isoperimetric inequality is an {\em isoprobabilistic inequality} when it shows that the expected value of a geometric random variable defined for a planar domain realizes an extrema when the domain is the disk. These inequalities form a special case of the more general isoperimetry phenomena and contain important examples and non-examples. 
 
 As a preliminary example, we provide a isoprobabilistic interpretation of the Rayleigh-Faber-Krahn inequality. This inequality states that given a bounded domain in $\R^n$, the first Dirichlet eigenvalue is no less than the corresponding Dirichlet eigenvalue of a Euclidean ball with the same volume. If we consider the first Dirichlet eigenvalue of a domain in terms of the exponential rate of exit times of Brownian motion, then this gives an isoprobabilistic version of the Rayleigh-Faber-Krahn theorem. 
           
Another isoprobabilistic inequality appears in Sylvester's 4-point problem~\cite{NP}, which studies the probability that the convex hull of four points chosen at random in a planar region is a quadrilateral. For this problem, Blaschke showed that the disk maximizes the probability modulo affine transformations ~\cite{WB}. As such, Blaschke's result can be thought of as an affine geometric version of an isoprobabilistic inequality. For more examples, we refer the reader to the works of Paouris and Pivovarov \cite{PP} as well as Bauer and Schneider \cite{CBRS} (who refer to such inequalities as ``extremal problems").
     
    
  
  However, such inequalities are not universal. There are natural geometric problems, such as the so-called ``grass-hopper problem," for which the disk is not a maxima \cite{GK}. This problem was featured in a recent FiveThirtyEight Riddler \cite{538}. It is of interest to determine conditions on a geometric problem so that an isoprobabilistic inequality holds. 

    Apart from isoprobabilistic inequalities, random triangles have been studied in various spatial processes and applications. One classic example is from Broadbent's 1980 paper ``Simulating the Ley-Hunter," \cite{SB} which statistically tested the hypothesis that ancient megalithic sites were built along ley-lines. That is to say, giant stone structures in England such as Stonehenge were purposely placed collinearly. Broadbent was unable to calculate precise distributions of random triangles, so he used computer simulations instead. His results suggested that although there was greater collinearity than expected, it was probably due to the clustering of the megalithic sites.

 D.G. Kendall \cite{DK} further studied the problem, and used a method similar to Crofton's differential equation to understand the distribution of the shape of the triangles when three points are chosen from a convex region. He was able to calculate this distribution explicitly when the region is a disk and show that the distribution of shapes is very close to being uniformly distributed. The importance of this is that it allows collinearity tests to be applied assuming a uniform prior, without introducing too much error. As one deviates from a disk, the assumption of uniformity introduces more error, and may fail to be useful in those cases.

         \subsection{A Generalization of Hall's Conjecture}
         
        For simplicity, all the results in this paper are in terms of Hall's original conjecture on acute triangles. However, there is a natural generalization which we refer to as the Strong Hall's conjecture. 
       \begin{conjecture}[Strong Hall's Conjecture]
\begin{equation}
F_S(\phi) \leq F_{n-{\rm ball}}(\phi).
\end{equation}
 \end{conjecture}
 
 Hall's paper uses $P(S)$ to be $F_S(\frac{\pi}2)$, so the original Hall's conjecture is the sub-conjecture that $F_S(\frac{\pi}2) \leq F_{n-{\rm ball}}(\frac{\pi}2)$. It is worth noting that the Strong Hall's Conjecture fails if $S$ is allowed to be non-convex. For instance, if we pick three points uniformly from three disks centered at the vertices of a large equilateral triangle, then there is a relatively large probability of forming a triangle that is very close to equilateral. Our work can be extended to prove a local version of the Strong Hall's conjecture conjecture as well as an estimate on $F_S(\phi)$ in terms of the isoperimetric ratio.

\section{The disk is a weak local maximum}

In this section, we present a proof of Theorem 1, which we restate here for convenience.

\begin{theorem*}
Let $S$ be a bounded convex subset of $\R^2$ and pick three points uniformly i.i.d. from $S$. We define $p(S)$ as the probability that these three points form an acute triangle.  The disk is a weak local maximum for the function. 
\end{theorem*}


 \begin{proof}

The general strategy is to show that the second variation of the probability is strongly negative.

\subsection{Some definitions}

In order to define the probability of choosing a convex triangle in a particular region $S$, we consider the following function.

\[ f(x,y,z) =   \left\{
\begin{array}{ll}
    1 & \triangle(x,y,z) \textrm{ acute} \\
    0 & 0\textrm{ otherwise} \\
\end{array} 
\right. \]

Then, the probability of choosing an acute triangle is 
\[ p(S) =  \frac{\int_S \int_S \int_S f(x,y,z)\,  dx \,  dy \,  dz}{V(S)^3} \]

This definition should be intuitively reasonable; it is the total mass of acute triangles in $S^3$ divided by the total volume of $S^3$. Note that $p(S)= F_S(\frac{\pi}2)$, by definition.

\subsection{The set-up}
Consider a $C^2$ family of convex regions $S(t)$ with $S(0)$ a unit disk centered at the origin. By applying a time dependent similarity, we assume that the volume remains constant and that it is centered at the origin throughout the variation. We also want to assume that the variation is not identically zero at $t=0$.  
For small $t$, this ensures we have enough regularity to write $S(t)$ in terms of polar coordinates as $S(t) = \{ re^{i \theta} \in \mathbb{R}^2 | r < r(\theta, t)\}$. Note that $r(\theta, 0) \equiv 1$.

We consider the density $\mu(\theta, t)$ (which is of mixed sign), so that 
$$r(\theta, t) = 1 + \mu(\theta, t)t.$$ 
Then, the second order Taylor expansion of $r(\theta, t)$ at $0$ is 
$$1 + t \mu(\theta, 0) + \frac{t^2}{2} \frac{d \mu(\theta, 0)}{ d t}.$$
As we have assumed that the variation is non-zero, $\mu(\theta, 0) \not \equiv 0$. Since $S(t)$ has constant volume to first order, $\mu(\theta, 0)$ has total integral zero.

There are a few technical points to make before proceeding.

\begin{enumerate}
\item We do not linearize the variation. If we do so, we cannot force the second variation of the area of the region to vanish. When we prove that the ball is a local maximum in the $C^0$ norm (Theorem 3), we will linearize the variation in a particular way.
\item Generally, one considers the variation as a function, and not in terms of a measure. However, as a historical connection to the original approach of moving manifolds and to make the Fourier analysis more natural, we instead consider $\mu(\theta, t)$ as a time-dependent density on the boundary of our region. We will often abuse notation and ``draw points" from the density $\mu$. In practice, this corresponds to drawing points proportional to the density $ | \mu |$ and then using the fact that $S$ is shrinking at the points where $\mu$ is negative. Finally, we will often omit either the $\theta$ or $t$ from $\mu(\theta, t)$ for conciseness.
\item The most natural assumption on the regularity is that the variation is $C^2$ smooth. The minimal regularity necessary is that $\frac{\partial^2 r}{\partial t^2}$ is $L^2$ in $\theta$ and $\frac{\partial r}{\partial t}$ is $L^2$ in $\theta$ and is uniformly integrable in $t$ for fixed $\theta$.
\end{enumerate}

 \subsection{Calculating the first and second variations}

Since we have fixed the volume throughout the variation, we calculate the variation of the total mass of acute triangles. We use polar coordinates in order to simplify the calculation.


\begin{flalign*}
 &M =  \int_{S \times S \times S} f(x,y,z)  \,dx \, dy \, dz   =  \\
 &\int_0^{2 \pi} \int_0^{2 \pi} \int_0^{2 \pi} \int_0^{r(\theta_1, t)}  \int_0^{r(\theta_2, t)}  \int_0^{r(\theta_2, t)}  f(r_1e^{i \theta_1},r_2e^{i \theta_2},r_3e^{i \theta_3}) r_1 r_2 r_3 \, dr_1 \,dr_2\, dr_3  \,d \theta_1\, d \theta_2\, d \theta_3.
\end{flalign*}

\noindent The first variation of $f$ is the following:
\begin{flalign*}
 \frac{d}{dt} \left(\int_{S \times S \times S} f(x,y,z)  \,dx \, dy \, dz \right)& = & 
\end{flalign*}
\begin{flalign*}
 \frac{d}{dt} \left(  \int_0^{2 \pi} \int_0^{2 \pi} \int_0^{2 \pi}  \int_0^{r(\theta_3, t)}   \int_0^{r(\theta_2, t)}  \int_0^{r(\theta_1, t)}  f(r_1e^{i \theta_1},r_2e^{i \theta_2},r_3e^{i \theta_3}) r_1 r_2 r_3 \,dr_1 \,dr_2 \,dr_3  \,d \theta_1 \,d \theta_2\, d \theta_3  \right)
\end{flalign*}
\begin{flalign*}
= 3 \int_0^{2 \pi} \int_0^{2 \pi}  \int_0^{2 \pi}   \int_0^{r( \theta_3, t)}  \int_0^{r( \theta_2, t)} &  f(r( \theta_1, t) e^{i \theta_1},r_2e^{i \theta_2},r_3e^{i \theta_3}) r(\theta_1 ,t) \mu(\theta_1, t) r_2 r_3\, dr_2\, dr_3\,  d \theta_1\, d \theta_2 \,d \theta_3.
\end{flalign*}

For convenience, we now prove that the disk is a critical point of the probability. Although Hall's work does not explicitly state it, his approach can be used to show this lemma. From personal correspondence, he was aware of this fact, so we do not claim the following result as new.
 
 \begin{lemma}
 The disk is a weak critical point of the probability functional.
 \end{lemma}
 
 \begin{proof}
 We use the variational formula to compute the first variation of $M$ at $t=0$.
 
 \begin{flalign*}
 \frac{d}{dt} \left(\int_{S \times S \times S} f(x,y,z)  \,dx \, dy \, dz \right)\Bigg|_{t=0} &  &
\end{flalign*}
\begin{flalign*} 
= 3 \int_0^{2 \pi} \int_0^{2 \pi}  \int_0^{2 \pi}   \int_0^{1}  \int_0^{1} &  f(r(0, \theta_1) e^{i \theta_1},r_2e^{i \theta_2},r_3e^{i \theta_3}) \mu(\theta_1, 0) r_2 r_3\, dr_2\, dr_3\,  d \theta_1\, d \theta_2 \,d \theta_3  & &
\end{flalign*}
\begin{flalign*}
= 3 \int_0^{2 \pi} \int_0^{2 \pi}  \int_0^{2 \pi}   \int_0^{1}  \int_0^{1} &  f(e^{i \theta_1},r_2e^{i \theta_2},r_3e^{i \theta_3})  r_2 r_3\, dr_2\, dr_3\, d \theta_2 \,d \theta_3 \, \mu(\theta_1, 0) d \theta_1. & &
\end{flalign*}

 By rotational symmetry of the circle, the following term is invariant in $\theta_1$:
  \begin{flalign*}
\int_0^{2 \pi}  \int_0^{2 \pi}   \int_0^{1}  \int_0^{1} &  f(e^{i \theta_1},r_2e^{i \theta_2},r_3e^{i \theta_3})  r_2 r_3\, dr_2\, dr_3\, d \theta_2 \,d \theta_3 
 \end{flalign*}
 
We call this term $\mathcal{A}$  as it is the mass of acute triangles when we fix one point on the boundary and choose the other two points uniformly at random. We can now substitute it back into our variational equation.  
 
$$
 \frac{d}{dt} \left( \int_{S \times S \times S} f(x,y,z)  \,dx \, dy \, dz \right) _{t=0}    
= 3 \int_0^{2 \pi} \mathcal{A} \mu(\theta_1, 0) \, d \theta_1  = 0.
$$

The final equality holds as the total integral of $\mu(\theta, 0)$ is zero since the variation has constant volume.

\end{proof}



 The second variation of $M$ is the following:

\begin{flalign}\label{eq:2nd var} \nonumber 
 \twodt \left( \int_{S \times S \times S} f(x,y,z)  \,dx \, dy \, dz \right) = &  & \\
\nonumber 3 \int_0^{2 \pi}  \int_S  \int_S &  f(r( \theta_1, t) e^{i \theta_1},y,z) \dt  \mu(\theta_1, t) r(\theta_1, t) \,dy \, dz \, d \theta_1 & & \\
\nonumber +  6 \int_0^{2 \pi} \int_0^{2 \pi} \int_S &  f(r(\theta_1, t) e^{i \theta_1},r(\theta_2, t) e^{i \theta_2},z) \mu(\theta_1, t) \mu(\theta_2, t) r(\theta_1, t) r(\theta_2, t)  \, dz \,  d \theta_1 \, d \theta_2 & & \\
\nonumber +  3 \int_0^{2 \pi} \int_S \int_S  \frac{\partial}{\partial r}& f(r(\theta_1, t) e^{i \theta_1},y,z) \mu(\theta_1, t)^2 r(\theta_1, t) \, dy \,d z \, d \theta_1 & & \\
\nonumber +  3 \int_0^{2 \pi} \int_S \int_S  & f(r(\theta_1, t) e^{i \theta_1},y,z) \mu(\theta_1, t)^2 \, dy \,d z \, d \theta_1 & & 
\end{flalign}

 Since $f$ is an indicator function, $\frac{\partial f}{\partial r}$ must be interpreted distributionally. We address this issue in detail when we analyze this term. We now analyze each of these four terms.

\subsection{The first term and fourth term}

Since the disk is a critical point, at $t=0$ (when $r(\theta, t) \equiv 1$), the following is independent of $\theta_1$: 
$$ \int_{S(0)}  \int_{S(0)}  f(r( \theta_1, t) e^{i \theta_1},y,z) \,dy\, dz$$

As a consequence of the results in \cite{GH}, this term is equal to $M$. As such, we have the following: 

\begin{eqnarray*}
 3 \int_0^{2 \pi}  \int_{S(0)}  \int_{S(0)}  f(e^{i \theta_1},y,z) \dt \mu(\theta_1, t) \,dy \,dz \,  d \theta_1 &  = &  3 \pi^2 \left( \frac{4}{\pi^2}-\frac{1}{8}\right) \int_0^{2 \pi} \dt \mu(\theta_1, t) \, d \theta_1 \\
\end{eqnarray*}


Since the volume is constant to second order, we have the following equality.

\begin{flalign*}
0  = \twodt V |_{t=0} = \twodt \left( \int_0^{2 \pi} r(\theta, t)^2 \, d \theta \right)_{t=0}
= \twodt \left(\int_0^{2 \pi} \left( 1+\mu(\theta, 0) t+ \frac{d \mu(\theta, 0)}{dt} \frac{t^2}{2} \right)^2 \, d \theta \right)_{t=0}
\end{flalign*}
\begin{flalign*}
= \left(\int_0^{2 \pi} 2\frac{d \mu}{dt} + 2 \mu^2 + 6 \mu \frac{d \mu}{dt} t + 3 \left(\frac{d \mu}{dt} \right)^2  t^2 \, d \theta\right)_{t=0}
=  \int_0^{2 \pi} \left(2\frac{d \mu}{dt} + 2 \mu^2 \, d \theta \right). & &
\end{flalign*}

Thus, we can rewrite the previous identity.

\begin{equation}\label{eq:1st term}
3\pi^2 \left( \frac{4}{\pi^2}-\frac{1}{8}\right) \int_0^{2 \pi} \dt \mu(\theta_1, 0) \, d \theta_1  =  -3\pi^2 \left( \frac{4}{\pi^2}-\frac{1}{8}\right) \int_0^{2 \pi} \mu(\theta,0)^2 \, d \theta_1	
\end{equation}

However, this exactly cancels out the first term in the second variation formula.

\begin{eqnarray*}
 3 \int_0^{2 \pi} \int_{S(0)} \int_{S(0)}   f(e^{i \theta_1},y,z) \mu(\theta_1,0)^2 \, dy \,d z \, d \theta_1 &=& 3 \int_0^{2 \pi} \int_S \int_S   f(e^{i \theta_1},y,z)  \, dy \,d z \, \mu(\theta_1,0)^2 d \theta_1 \\
 \end{eqnarray*}
 The inner integral is independent of $\theta_1$, and so we find the following:
\begin{eqnarray*}
  3 \int_0^{2 \pi} \int_{S(0)} \int_{S(0)}   f(e^{i \theta_1},y,z) \mu(\theta_1,0)^2 \, dy \,d z \, d \theta_1 & = &  3\pi^2 \left( \frac{4}{\pi^2}-\frac{1}{8}\right) \int_0^{2 \pi}  \mu(\theta_1,0)^2 d \theta_1
\end{eqnarray*}

Therefore, the first and fourth terms exactly cancel when $t=0$.

\subsection{The second term}

We now consider the second term:
\begin{equation}\label{eq:2nd term}
6 \int_0^{2 \pi} \int_0^{2 \pi} \int_S f(r(\theta_1, t) e^{i \theta_1},r(\theta_2, t) e^{i \theta_2},z) \mu(\theta_1, t) \mu(\theta_2, t) r(\theta_1, t) r(\theta_2, t) \, dz \,  d \theta_1\, d \theta_2
\end{equation}

This term is non-trivial in that it depends on the particular choice of $\mu(\theta, 0)$. For the rest of this calculation, we restrict our calculation to $t=0$, when the region is a disk. 
This term corresponds to the case in which we pick two points $X$ and $Y$ on the boundary i.i.d with respect to the density $ \mu$ and look at the mass of convex triangles when we do this. By symmetry, this mass only depends on the angle $\theta$ between $X$ and $Y$. Note that in Figure 1, we will have an acute triangle if the third point, chosen uniformly at random within the larger circle, falls in the shaded area between the two parallel lines and outside of the smaller circle. The mass of acute triangles is simply the area of this region.

\begin{center}
\includegraphics[width=150mm,scale=0.5]{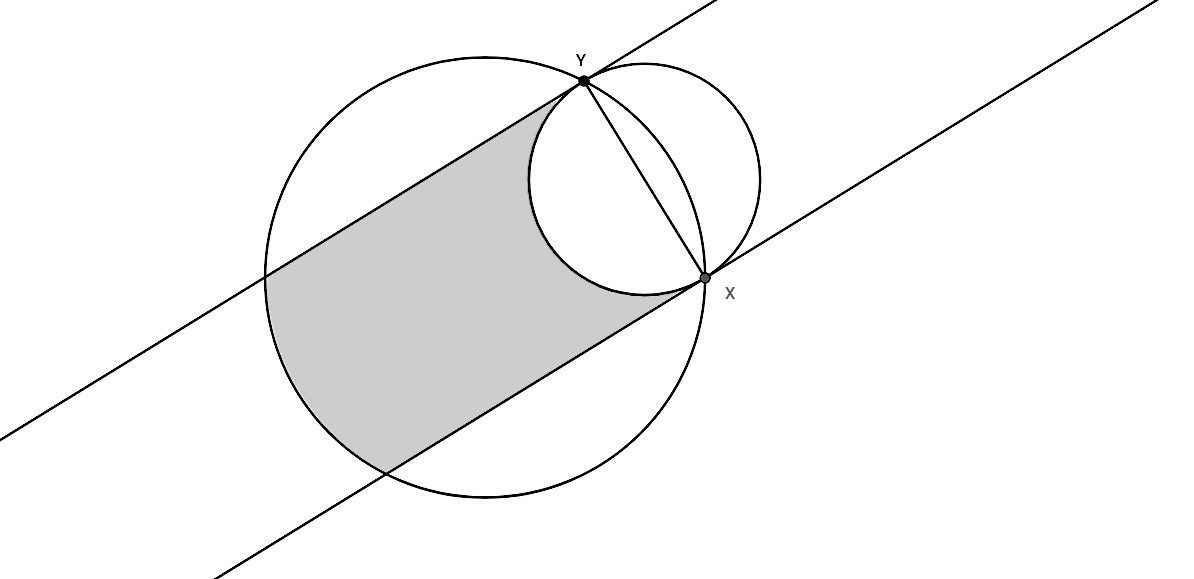}
Figure 1. 
\end{center}

We let $A_2(\theta)$ be the mass of acute triangles with angle $\theta$ between $X$ and $Y$. An exercise in geometry (done in \cite{GH}) shows that:

\[ A_2(\theta) =   \left\{
\begin{array}{ll}
     \frac{ \theta}{2}+\frac{3 \sin( \theta)}{2}- \frac{\pi}{4}+ \frac{\pi \cos(\theta)}{4} &  0<\theta< \pi \\
   A_2(2 \pi - \theta) =  -\frac{ \theta}{2} - \frac{3 \sin( \theta)}{2} + \frac{3 \pi}{4} + \frac{\pi \cos(\theta)}{4} &  \pi<\theta< 2 \pi \\
\end{array} 
\right. \]

Thus, we have that ~\eqref{eq:2nd term} is proportional to 
$$\int_0^{2\pi} \int_0^{2 \pi} \mu(u) \mu(u+ \theta) du~ A_2(\theta) d\theta.$$

We consider the Fourier series for $A_2(\theta)$. As $A_2(\theta)$ is even, this only involves $\cos(n\theta)$ terms. 

\begin{lemma}[The Fourier coefficients of $A_2(\theta)$]

$$ \frac{1}{2 \pi} \int_0^{2 \pi} A_2(\theta) d\theta = \frac{3}{\pi}$$
$$ \frac{1}{2 \pi} \int_0^{2 \pi} A_2(\theta) \cos(\theta) d\theta = \frac{1}{8 \pi} (-8 + \pi^2)$$
For $n \geq 2$, we have
$$ \frac{1}{2 \pi} \int_0^{2 \pi} A_2(\theta) \cos(n \theta) d\theta =\frac{1}{2 \pi} \frac{(-1)^n (-2 - 4 n^2 + (2 - 8 n^2) (-1)^n )}{2 n^2 (-1 + n^2)}$$

In particular, zero-th and first Fourier coefficients of $A_2(\theta)$ are positive and the rest are negative.

\end{lemma}

To prove this lemma, we can directly compute the integrals. 
Using this observation, we now compute $\int_0^{2\pi} \int_0^{2 \pi} \mu(u) \mu(u+ \theta) du~ A_2(\theta) d\theta.$
 We let $R_\mu( \theta ) = \int_0^{2 \pi} \mu(u) \mu(u+ \theta) du $ be the non-normalized autocorrelation of $\mu$. Note that $R_\mu$ is $L^\infty$ (hence also $L^1$ and $L^2$) since $\mu$ is $L^2$. Therefore, we can use Fubini's theorem, Fourier analysis and our previous work to show the following.

\begin{flalign*}
\int_0^{2\pi} \int_0^{2 \pi} \mu(u) \mu(u+ \theta) du~ A_2(\theta) \, d\theta = \int_0^{2\pi}  R_\mu(\theta) A_2(\theta)\, d\theta & &
\end{flalign*}
\begin{flalign*}
=  \int_0^{2\pi}  R_\mu(\theta)  \sum_{n=0}^\infty a_n \cos(n \theta)\, d\theta =
 \sum_{n=0}^\infty a_n  \int_0^{2\pi} R_\mu(\theta) \cos(n \theta)\, d\theta 
= 2 \pi \sum_{n=0}^\infty a_n  \mathcal{F}[ R_\mu](n). & &
\end{flalign*}

In order to write the right hand term more explicitly in terms of $\mu$, we use the following consequence of the autocorrelation theorem.

\begin{obs*}
If the Fourier series of $\mu(\theta,0)$ is given by $\mu(\theta, 0) = \sum_{k=0}^\infty c_k  \cos (k \theta) + d_k \sin(k \theta)$, then the Fourier transform of $R_\mu$ is the following:
\[ \mathcal{F}[R_\mu] (n) = c_n^2 + d_n^2 \]
\end{obs*}

Substituting this identity into the autocorrelation term, we find the following:

\begin{eqnarray*}
6 \int_0^{2 \pi} \int_0^{2 \pi} \int_S f(e^{i \theta_1},e^{i \theta_2},z) \mu(\theta_1, 0) \mu(\theta_2, 0) \, dz \,  d \theta_1\, d \theta_2 \\
\propto \sum_{n=0}^\infty a_n (c_n^2+d_n^2)
\end{eqnarray*}

To give some intuition, note that since $\mathcal{F}[R_\mu] (n) \geq 0$ for all $n$, then we have $ \sum_{n=0}^\infty a_n \mathcal{F}[R_\mu] (n) >0$ implies that either $\mathcal{F}[R_\mu] (0)$ or $\mathcal{F}[R_\mu] (1)$ are non-zero. We will later show that if our variation preserves the volume, $\mathcal{F}[R_\mu ] (0) = 0$ and if our variation does not translate the region, $\mathcal{F}[R_\mu] (1) = 0$. 
The fact that $a_0$ is positive corresponds to the fact that the mass of acute triangles increases as one scales up the region. Later, there will be a negative term that exactly cancels out positive contribution from $a_1$, which corresponds to the fact that translations do not change the mass of acute triangles. The fact that $a_n$ is negative otherwise shows that other deformations decrease the mass of acute triangles.

\subsubsection{A quick calculation}

For the next step, it is useful to calculate this term precisely when $\mu = \cos(\theta) - \sin^2(\theta) t $ when $t=0$, as this allows us to skip several integrals.

\begin{eqnarray*}
6 \int_0^{2\pi} \int_0^{2 \pi} \mu(u) \mu(u+ \theta)\, du~ A_2(\theta)\, d\theta & = & 6 *2 \pi \sum_{n=0}^\infty a_n  \mathcal{F}[ R_{\cos(\theta)}](n) \\
& = & \frac{3}{4 \pi} (-8 + \pi^2) \pi^2 \\
\end{eqnarray*}

\subsection{The third term}

We now calculate the third term, which we write formally as the following.

\begin{equation}\label{eq:3rd term}
3 \int_0^{2 \pi} \int_S \int_S \frac{\partial}{\partial r} f(r( \theta_1, t) e^{i \theta_1},y,z) (\mu(\theta_1))^2 ~dy \,d z\, r( \theta_1, t) d \theta_1  
\end{equation}

 This term corresponds to times where we pick one point on the boundary and two in the interior that form a right triangle. In this case, moving the boundary changes this triangle from obtuse to acute or vice-versa.

 \begin{center}
\includegraphics[width=125mm,scale=0.5]{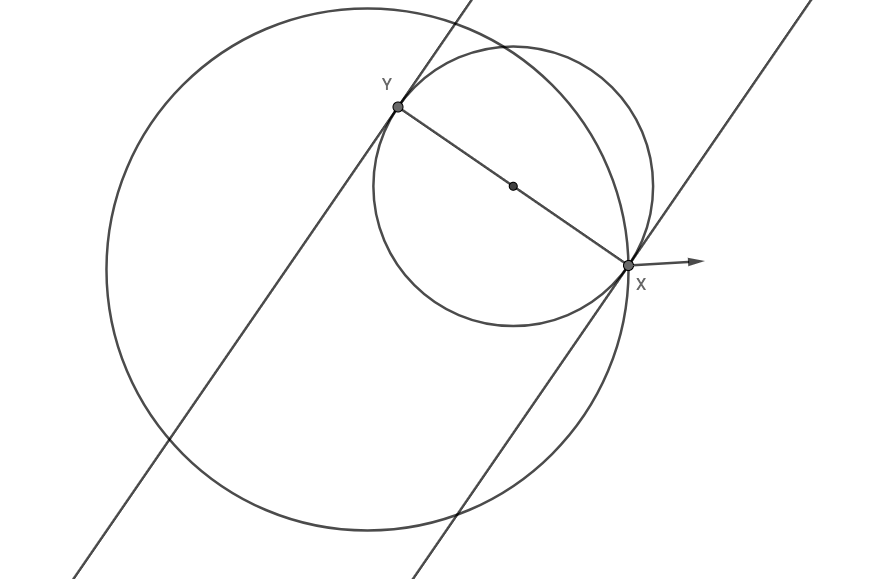}
Figure 2. 
\end{center}

 In Figure 2, $X$ is chosen on the boundary with respect to $\mu$ and the variation moves $X$ outward (as indicated by the vector). We then pick $Y$ inside the larger disk and a point $Z$ (not shown) which lies on the smaller circle or one of the two parallel lines. Due to the movement of $X$, the triangle $\triangle(XYZ)$ will change from obtuse to acute or vice-versa.
 
 To write this down classically, we consider the curve $X(t)=X_0 + Vt$ where $X_0$ is a base point and $V$ is a chosen vector. Then, we interpret $ \frac{\partial}{\partial V} f(X(t),y,z)$ distributionally to satisfy
 \begin{eqnarray*}
\int_S \int_S \frac{\partial}{\partial V} f(X(t),y,z) ~dy \,d z  = \dt \int_S \int_S f(X(t),y,z) ~dy \,d z.
\end{eqnarray*}
 
  If we wanted to avoid distributional derivatives, we can evaluate the second integral classically. 
   This would result in 3 separate quadruple integral over a complicated geometric set. 
    For our purposes, it is much more elegant and less tedious to leave it as a distributional derivative, and we will do so.


  We first reduce this integral to the following:

$$   3 \int_0^{2 \pi} \frac{\partial}{\partial r} \int_S \int_S  f(r( \theta_1, t) e^{i \theta_1},y,z)  \,dy\, d z (\mu(\theta_1))^2 \, r( \theta_1, t) d \theta_1  $$

Furthermore, by the symmetry of the circle, we note that the following is independent of $\theta_1$:
 $$ \frac{\partial}{\partial r} \int_S \int_S  f(r( \theta_1, t) e^{i \theta_1},y,z)  \,dy\, d z $$

  We want to compute the above integral. To do so directly would require involved geometric integrals. Therefore, we will use the following observation. We observe that if 
$$\mu = \cos(\theta)-\dfrac{t \sin^2(\theta)}{(1-t^2 \sin^2(\theta))^{1/2}},$$
this corresponds to a shift of the disk. At $t=0$, $\mu \approx \cos(\theta)-t \sin^2(\theta)$ to second order and so we set $\nu = \cos(\theta)-\sin^2(\theta)t$. As a result,
$$\twodt P(S_\nu) = 0.$$
For this to be true,
   $$ 3 \int_0^{2 \pi} \frac{\partial}{\partial r} \int_S \int_S  f(r(\theta_1, t) e^{i \theta_1},y,z)  \,dy\, d z \cos^2(\theta_1) \, d \theta_1$$ must exactly cancel out with the term from ~\ref{eq:2nd term}.
   
Calculating these directly, we find the following
$$
6 \int_0^{2\pi} \int_0^{2 \pi} \cos(u) \cos(u+ \theta) du\, A_2(\theta) \,d\theta = \frac{3}{4 \pi} (-8 + \pi^2) \pi^2. 
$$

   \noindent Therefore, $$ 3 \frac{\partial}{\partial r} \int_S \int_S  f(r( \theta_1, t) e^{i \theta_1},y,z) \, dy\,d z =   -\frac{3}{4 \pi} (-8 + \pi^2) \pi^2  $$
    For convenience, we define $-L=  \frac{3}{4 \pi} (  8 - \pi^2 ) \pi^2$. Using the Plancherel identity, this shows the following identity:

\begin{flalign*}
    3 \int_0^{2 \pi} \frac{\partial}{\partial r} \int_S \int_S  f(r( \theta_1, t) e^{i \theta_1},y,z)\,  dy\, d z \, (\mu(\theta_1))^2 \, d \theta_1 = -L \int_0^{2 \pi}   \mu(\theta_1)^2 \, d \theta_1  = -L  \sum_{k=0}^\infty c_k^2 + d_k^2
    \end{flalign*}

\subsection{Combining all the terms}

To finish the proof, recall that we have decomposed $\mu(\theta,t)$ in terms of its Fourier series as follows:

  \[ \mu(\theta,t)= \sum_{k =0}^\infty c_k \cos(k \theta) + d_k \sin( k \theta) \]

We have that that $c_0(0) = 0$ since the total integral of $\mu(\theta, 0)$ is zero, and that $d_0(0)=0$ as $\sin(0 \cdot \theta) \equiv 0$. Thus, for any measure $\mu$ that preserves the volume, we have: 

\begin{equation*}
 \twodt \int_{S \times S \times S} f(x,y,z)\, dx\, dy\, dz =  \sum_{ k=1}^\infty (a_k -L) (c_k^2 + d_k^2)  \leq  0 \\
    \end{equation*}
    
The above equality only holds for the sum, and cannot be decomposed into a term by term equality. This is because we have that $$\int_0^{2 \pi} \frac{d \mu}{dt} \, d \theta ~|_{t=0} =  \int_0^{2 \pi} \mu^2 \, d \theta$$ but we don't have any information about how $\dfrac{d \mu}{dt}~ \bigl |_{t=0}$ decomposes as a Fourier series.

Equality holds only when $c_1(0)$ and $d_1(0)$ are the only non-zero terms, in which case $\mu(0) =  c \cos(\theta) + d \sin(\theta)$. In this case, the variation translates the disk to second order, and so our assumption that the disk is not translated is equivalent to the fact that $c_1(0)=d_1(0) = 0$.

To finish the argument, suppose that $S(t)$ is a $C^2$ family of convex regions that have constant volume and which are not translated. Suppose further that $\mu(\theta, 0) \not \equiv 0$. Then, since $\mu(\theta, 0)$ is $L^2$ in $\theta$ we can use the Plancherel theorem to show the following:

\begin{equation*}
 \twodt \int_{S \times S \times S} f(x,y,z)\, dx\, dy\, dz = \sum_{k=2}^\infty (a_k -L) (c_k^2 + d_k^2)   
									 < -L \| \mu (\theta, 0) \|^2_{L^2}\, .
 \end{equation*}

  Thus the second variation is strongly negative, and so the disk is a weak local maximum of the probability functional (see Theorem 2 of Chapter 5 in \cite{GF}).


\end{proof}
 
 It is worth noting that the same argument can be used to prove a weak local version of the Strong Hall's Conjecture. We define $A(\theta, \phi)$ as the mass of angles whose largest angle is no larger than $\phi$, given two points on the boundary separated by an angle $\theta$  The key lemma that must be proven is that the zeroth and first Fourier coefficients of $A(\theta, \phi)$ (in terms of $\theta$) are positive while all the rest are negative. Using Mathematica and the results of \cite{ES}, we checked this numerically for $\pi/2 \leq \phi < \pi$. 

 \section{The ball is a weak local maximum}
In $\mathbb{R}^3$, we can prove that the ball is a weak local maximum for the probability functional. The proof is similar to that of Theorem 1 except with additional technical details. The integrals are more demanding and we used Mathematica to evaluate them. Furthermore, we used representation theory to generalize the convolution theorem and Plancherel identity for $SO(3)$. Precisely, we prove the following theorem.

\begin{theorem*}
The ball is a weak local maximum for $p$. More precisely, given a smooth one parameter family of convex regions $S(t)$ where $S(0)$ is the ball, $p(S(t))$ has a local maximum at $0$.
\end{theorem*}

 \begin{proof}

 We provide some of the main details for the $3$-dimensional case. The rest of the proof is exactly the same as in $\R^2$, so we will only include the parts that are different. 

 \subsection{The second variational formula}

We start by considering a smooth one-parameter family of convex regions which preserves the volume and does not translate the region. The first variation is essentially identical to that for the disk, as is the proof that the ball is a weak critical point. 

  The second variation formula is also nearly identical to the two-dimensional case. However, here the $\theta_i$ correspond to points on the unit sphere, which we denote by $\mathbb{S}^2$.
 
\begin{flalign*} \label{2nd var 3D}
 \twodt \int_{S \times S \times S} f(x,y,z)  \,dx \, dy \, dz  = &   \\
= 3 \int_{\mathbb S^2}  \int_S  \int_S &  f(\theta_1 ,y,z) r(\theta_1,t) \dt \mu(\theta_1, t) \,dy \, dz \, d \theta_1 \\
+  6 \int_{\mathbb S^2} \int_{\mathbb S^2} \int_S &  f( r(\theta_1, t) \theta_1 , r(\theta_1, t) \theta_2,z) \mu(\theta_1, t) \mu(\theta_2, t) r(\theta_1,t) r(\theta_2,t) \, dz \,  d \theta_1 \, d \theta_2 \\
+  3 \int_{\mathbb S^2} \int_S \int_S  \frac{\partial}{\partial r}& f( r(\theta_1, t) \theta_1,y,z) r(\theta_1,t) (\mu(\theta_1))^2 \, dy \,d z \, d \theta_1  \\
+3 \int_{\mathbb S^2}  \int_S  \int_S &  f(\theta_1 ,y,z) \mu(\theta_1, t)^2 \,dy \, dz \, d \theta_1 
\end{flalign*}

 For the second term, we write $\mu = \sum_{i \geq 0, |m| \leq i} c_i^m Y_m^i$ and consider its autocorrelation:
  $$R_\mu(g) = \int_{x \in \mathbb{S}^2} \mu(x) \mu(xg) dS$$
  
  Here, $g$ is an element of $SO(3)$, which acts on $x$ by right rotation. Noting that the stabilizer of each point is $SO(2)$-subgroup of $SO(3)$, we consider the Fourier series of $R_\mu(g)$ in the symmetric space $g \in SO(3)/SO(2) \simeq \mathbb S^2$. Put more simply, we write $R_\mu$ as the sum of spherical harmonics:
 \[ R_\mu = \sum_{i,m} b_i^m Y_m^i \]  
  
  We verify several properties about the decomposition of $R_\mu$.
 
  \begin{lemma}[Properties of the Fourier series of an autocorrelation]
 
Consider the autocorrelation $R_\mu = \sum_{i,m} b_i^m Y_m^i$. Then the coefficients $b_i^m$ have the following properties.
\begin{enumerate}
\item For all $i, m$, $b_i^m \geq 0$ .
\item If $c_i^m \neq 0$, $b_0^m > 0$.
\item The $c_i^m$ terms do not contribute to $b_j^n$ for $n \neq m$. In particular, none of the higher $c_i^m$ terms contribute to $b_0^0$ or $b_0^1$, which are the only positive terms in the Fourier series of $A_3(\theta)$.
\end{enumerate}

\end{lemma}

This lemma follows directly from the convolution theorem on $SO(3)$. We include a more detailed discussion in the appendix.

  \subsection{Calculating the Fourier coefficients in three dimensions}

We now define $A_n(\theta)$ as the total mass of acute triangles when two points are chosen on the boundary of the $n$-ball with $\theta$ the minimal angle between them. From \cite{GH}, this is given by the following formula:
\begin{equation} \label{A 3D}
A_n(\theta)= \frac{\pi ^{\frac{n}{2}-\frac{1}{2}} \left(\int_{\frac{\theta}{2}}^{\frac{\pi }{2}} \sin ^n(t) \, dt+\int_0^{\frac{\theta}{2}} \cos ^n(t) \, dt\right)}{\Gamma \left(\frac{n}{2}+\frac{1}{2}\right)}-\frac{\pi ^{n/2} \left(\sin ^n\left(\frac{\theta}{2}\right)+1\right)}{\Gamma \left(\frac{n}{2}+1\right)}
\end{equation}


This is overusing $\theta$: it is simultaneously the minimal angle between two points on the sphere and the $\theta$ parameter in spherical coordinates. However, if we consider the first point to be the north pole of the sphere, then these two notions coincide exactly and the Fourier analysis is simpler without introducing yet another variable. We write $A_3(\theta)$ as the sum of spherical harmonics. To do this, we integrate $A_3(\theta)$ against spherical harmonics and show the following lemma.

\begin{lemma*}[Spherical harmonic decomposition of $A_3(\theta)$]
Let $Y^i_m$ be spherical harmonics where $m$ is the $\theta$ parameter and $i$ is the $\phi$ parameter.
Then we have the following:

   \begin{eqnarray*}
 \int_{S^2} A_3(\theta) Y_m^i(\theta, \phi ) \sin{\theta} \,d \theta \hspace{.3in}
  \begin{cases}
\, >0 \textrm{ for } m = 0,1 , i = 0  \\
\, = 0 \textrm{ for } i \neq 0 \\
 \, \leq 0 \textrm{ otherwise}  \\
\end{cases} 
   \end{eqnarray*}
\end{lemma*}

  Note that $A_3(\theta)$ does not depend on $\phi$. Due to this $\phi$-invariance, we have that \[ \int_{\mathbb{S}^2} A_3(\theta) Y^i_m \,dS  =0 \textrm{ when  }i \neq 0. \]
  Therefore, the lemma reduces to a computation in Legendre polynomials. Ignoring some scaling factors, we must show the following lemma.
 
 \begin{lemma}
 
 Let $P_n(\cos (\theta ))$ be the $n$-th Legendre polynomial evaluated at $\cos (\theta )$. Then the following inequalities hold:
 
   \begin{eqnarray*} \int_0^\pi A_3(\theta) P_n(\cos (\theta )) \sin{\theta} \,d \theta
  \begin{array}{ll}
>0 \textrm{ for } n = 0,1  \\
  \leq 0 \textrm{ otherwise}  \\
\end{array} 
   \end{eqnarray*}

\end{lemma}

We postpone the proof of this lemma to the appendix.

 
   

 \subsection{Finishing the proof}
 
 We are now able to replicate the argument from $\R^2$. We expand $\mu(\theta, t)$ in terms of spherical harmonics:
  $$\mu(\theta, t) = \sum_{i \geq 2, |k| \leq i} c^i_k(t)  Y_i^k.$$
  
  The assumption that our variation does not translate the ball and preserves volume implies that the $Y_0^0$, $Y_1^0$, $Y_1^1$ and $Y_1^{-1}$ components of $\mu(\theta, 0)$ are zero.


The rest of the proof is exactly the same in the 2-dimensional case. The first term relates the $L^2$ norm of $\mu$ to $\frac{\partial \mu}{\partial t}$ and exactly cancels out the fourth term. The two lemmas proven in this section control the second term, and show that it is negative whenever $c_0^0$, $c_0^1$, $c_1^1$ and $c_{-1}^{1}$ are zero. When we integrate $R_\mu(g)$ against $A_3(\theta)$, only the zonal terms contribute and the previous lemma shows that they have the necessary signs. 
To calculate the third term, the trick to compute $-L$ for the third term is exactly the same as before (we can set $\mu(\theta, \phi, t) = \cos (\theta) - \frac{t \sin^2(\theta)}{\sqrt{1-t^2\sin^2(\theta)}}$ to obtain a translation). However, we must use a modified Plancherel identity for the last term (see ~\cite{AT}) to show that the decomposition in terms of spherical harmonics preserves $L^2$ norm. Putting this together, the second variation is once again strongly negative so long as the variation maintains the volume and does not translate the ball. This then shows that the ball is a weak local maximum.

 \end{proof}

 We believe that this analysis can be reproduced in higher dimensions as well. Spherical harmonics have been calculated for the $n$-sphere \cite{FE} and Hall's computation provides a formula in arbitrary dimensions. We have not done so due to a lack of computational resources. 

\section{The disk is a strong local maximum}

We now prove that the disk is a local maximum in the Gromov-Hausdorff topology. The general strategy is to obtain a type of $C^{2,1/2}$-estimate on the probability functional $p(S)$ for regions close to the disk. We can then use this estimate to find a barrier function for the probability functional. Because the proofs of some of the lemmas involve lengthy estimates which obfuscate the main strategy, we will give the main overview in this section, and postpone the proofs of these lemmas to the appendix. 

\begin{theorem*}
Let $S$ be a convex subset of $\R^2$ and pick three points uniformly i.i.d. from $S$. We define $p(S)$ as the probability that these three points form an acute triangle. The disk is a local maximum for this function. More precisely, there exists $\epsilon$ such that given a region $S$ with $d_{H.}(D,S) < \epsilon$, then $p(S) < p(D)$. 
\end{theorem*}

\begin{proof}

Let $S \subset \R^2$ be a bounded convex region with $d_{H.}(D,S) < \epsilon$. We assume that $\epsilon < \frac{1}{8 \sqrt{2 \pi}}$, to simplify later estimates.  As a result, $S$ necessarily contains the origin and we can write $S$ in terms of polar coordinates: $$S = \{ re^{i \theta} ~ | ~ r \leq 1 + g(\theta) \}.$$
Note that since $d_{H.}(D,S) < \epsilon$, this implies that $|g(\theta)| < \epsilon$. 

We refine this map by choosing a nice initial embedding of $S$. In particular, we want to choose an embedding for which the zeroth and first Fourier coefficients of $g(\theta)$ are zero. If the zeroth and first Fourier coefficients of $g(\theta)$ are non-zero and all of the others vanish, the second variation of the probability at $t=0$ is $0$, which complicates our estimates. To avoid this problem, we translate and dilate $S$ (doing so implicitly redefines $g(\theta)$), as shown in the following lemma. 


\begin{lemma}

There is a convex set $\bar S  = \{ re^{i \theta}~ |~ r \leq 1 + \bar g(\theta) \}$ with the following properties.

\begin{enumerate}
\item $\bar S$ is similar to $S$. 
\item $\bar S$ can be obtain by translating $S$ by no more than $3\epsilon$ and dilating $S$ by a factor between $1-3\epsilon$ and $1+3\epsilon$. Furthermore, $d_{H.}(S, \bar S) < 7 \epsilon$.
\item The function $\bar g(\theta)$ satisfy the following:
 $$ \int_0^{ 2 \pi} \bar g(\theta) d \theta = 0 \hspace{2cm} \int_0^{ 2 \pi} \bar g(\theta) e^{\ i \theta} d \theta = 0$$
\end{enumerate}
\end{lemma}


We postpone the proof of this lemma to the appendix. For the rest of the proof, we will utilize the following abuse of notation. Instead of writing $\bar S$, we always deal with the case where $S$ is embedded so that $g$ satisfies $ \int_0^{ 2 \pi}  g(\theta) d \theta = 0 $ and $ \int_0^{ 2 \pi}  g(\theta) e^{\ i \theta} d \theta = 0$. We will also assume that $S$ has Hausdorff distance from $D$ at most $8 \epsilon$. 

\subsection{The canonical homotopy}

We now construct a particular homotopy from the disk to $S$. 
 We denote this homotopy as $S(t)$ and refer to it as the {\em canonical homotopy}:

\begin{equation}
 S(t)= \{ re^{i \theta} \in \R^2 ~|~ r \leq 1 + \frac{t}{\|g\|_2} g(\theta) \}
 \end{equation}

Before going further, we note several properties of the canonical homotopy.

\begin{enumerate}
\item We do not make any smoothness assumptions on $S$, other than those that follow from convexity. As such, we do not have a point-wise estimate on $\mu(\theta,t) := \frac{g(\theta)}{\|g\|_2}$, only a $L^2$ estimate. If $S$ is not a disk, then $d_{H.}(S,D) >0$, and we can use the convexity of $S$ to establish a lower bound on $\|g\|_2$. To see this, one can use the convexity of $S$ to obtain a lower bound on the area of $S \Delta D$ in terms of $d_{H.}(S,D)$. This shows that the canonical homotopy is well defined when $S$ is not a disk. 
We also have an upper estimate $\|g \|_2 <  8 \sqrt{2 \pi} \epsilon$. By our assumption on $\epsilon$, this implies that $\|g \|_2 <1$.
\item  Note that whenever the $S$ is not a disk, we have that $S(0) =D$, $S = S(\|g\|_2)$ and $\mu(\theta, t) = \frac{g(\theta)}{\|g\|_2}$. 
\item As $d_{H.}(S(t),D)$ is increasing in $t$, we have that $d_{H.}(S(t),D)< 8 \epsilon$ for all $t < \|g\|_2$. 
\item Because $S$ is convex, $S(t)$ is convex for $0\leq t \leq \|g\|_2$. However, $S(t)$ may be non-convex for all $t > \|g\|_2$.
\item Observe that $\mu(\theta,t)= \frac{g(\theta)}{\|g\|_2}$ is independent of $t$, so we will often denote it as $\mu(\theta)$ for conciseness.
\end{enumerate}


\subsection{The variations of the probability}

For Theorem 1, we did not compute the second variation of the probability directly, as we assumed that the variation preserved area. The canonical homotopy does not preserve area and so we calculate the variations of the probability.


\[ p(S) :=  \frac{\int_S \int_S \int_S f(x,y,z) \,  dx \,  dy \,  dz}{Vol(S)^3} \]

For convenience, denote $M := \int_{S(t)} \int_{S(t)} \int_{S(t)} f(x,y,z) \,  dx \,  dy \,  dz$ and $V := Vol(S(t))$.
Using this notation, the first two variations of $p$ are the following:

\begin{flalign*}
 \frac{d}{dt} p = & \frac{dM}{dt} V^{-3} - 3 M V^{-4} \frac{dV}{dt}   
\end{flalign*}

\begin{flalign*}
 \twodt p = & \frac{d^2M}{dt^2} V^{-3} - 6 V^{-4} \frac{dM}{dt}  \frac{dV}{dt}  +12  M V^{-5} \left( \frac{dV}{dt} \right)^2 - 3 M V^{-4} \frac{d^2 V}{dt^2}  
\end{flalign*}

The formula for $ \twodt p$ can be simplified using the equation for $\dt p$.

\begin{flalign*}
 \twodt p = & \frac{1}{V^3} \left( \frac{d^2M}{dt^2} -  6 \frac{dV}{dt} \left( V^2  \frac{dp}{dt} + pV \frac{dV}{dt}  \right) - 3 pV^2 \frac{d^2 V}{dt^2} \right) \\
\end{flalign*}

We omit the proof that the disk is a critical point of the probability functional along the canonical homotopy, as it is exactly the same as the earlier proof of this fact.
We now calculate the first two variations of $V$. Recalling that we chose our embedding of $S$ to satisfy $ \int_0^{ 2 \pi} g(\theta) d \theta = 0$, we can calculate $V(t)$ explicitly. Simplifying and using the normalization $\| \mu \|_2 =1$, we obtain the following.

\begin{equation}\label{eq:Prob DE}
V^{3} \twodt p  =  \frac{d^2M}{dt^2} - 3p \pi - 6 \pi t  \frac{dp}{dt} -9 \pi  t^2 p \\
  -6 \pi t^3  \frac{dp}{dt} - \frac{15}{4} t^4  p - \frac{3}{2} t^5 \frac{dp}{dt} \\
\end{equation}

\subsection{Analyzing the differential inequality}

To understand solutions to \eqref{eq:Prob DE}, we use the following ansatz. For small $t$, all of the terms on the right hand side other than $\frac{d^2M}{dt^2} - 3p \pi$ are uniformly small. We will show that this term is negative, and so the entire equation is negative. This shows that $p(S(t))$ is strictly concave in a neighborhood of zero, which we use to prove the theorem.

 We will prove all of the estimates in terms of $S$. However, since the estimates depend only on the Hausdorff distance from the disk and the convexity of the region, they apply uniformly to $S(t)$ for all $t < \|g\|_2$.


To start, note that $\frac{d^2M}{dt^2}  - 3 p \pi <0$ for the disk (i.e., when $t = 0$). This follows from Theorem 1 and the fact that $\int_0^{2 \pi} g(\theta) e^{i \theta} \,d\theta = 0$, which shows the following: $$\left( \frac{d^2M}{dt^2}  - 3 p \pi \right) \Bigg |_{t=0} < -L=  \frac{3}{4 \pi} (  8 - \pi^2 ) \pi^2$$
 To show that this inequality remains true near the disk, we use the following lemmas, which show that $p$ is uniformly Lipschitz and $\frac{d^2M}{dt^2}$ is uniformly H\"older-$1/2$ continuous.

\begin{lemma}[The Lipschitz estimate on $p(S)$]
There exist uniform constants $ e,~ C > 0$ so that whenever $\epsilon < e$ and $d_{H.}(S,D) \leq \epsilon$, then $ \bigl\lvert  p(S) - p(D)  \bigl\rvert < C \epsilon $.
\end{lemma}

\begin{lemma}[The H\"older-$1/2$ estimate on  on $\frac{d^2M}{dt^2}$] 
There exist uniform constants $ e,~ C >0$ so that whenever $\epsilon < e$ and $d_{H.}(S,D) \leq \epsilon$, then $ \bigl\lvert \frac{d^2}{dt^2} M (S) - \frac{d^2}{dt^2} M (D) \bigl\rvert < C \epsilon^{1/2} $.
\end{lemma}
 
 Together, these immediately imply the following corollary.
 
 \begin{cor}
There exist uniform constants $C,~ t_0, \epsilon >0$, so that whenever $t < t_0$ and $d_{H.}(S,D) \leq \epsilon$, then $ | (\frac{d^2M}{dt^2} - 3p \pi) + L |< C t^{1/2} $ along the canonical homotopy.
\end{cor}

We now use this corollary to complete the proof. Note that the Lipschitz estimate controls $\frac{dp}{dt}$, which confirms our ansatz that the terms with a factor of $t$ are uniformly small for small $t$. The proofs of these lemmas are composed of careful, tedious estimates on the Holder constant of each the terms in the second variation. We postpone the proofs of the two lemmas to the appendix.

From the preceding H\"older estimate, we can find a uniform $\epsilon_0$ so that whenever $d_{H.}(S,D) \leq \epsilon_0$, we have that $\frac{d^2M}{dt^2} - 3p \pi \leq - \frac{L}{2}$ for all $t$ satisfying $0 \leq t \leq \|g\|_2$ along the canonical homotopy $S(t)$.
Note that we can choose $\epsilon_0$ so that this inequality holds for all $t$ up to $ \|g\|_2$ because $t_0$ in the corollary is uniform, and $\sqrt{2 \pi} d_{H.}(S,D) \geq \|g\|_2$. In other words, by taking $\epsilon_0$ sufficiently small, we can force $t_0 \geq \|g\|_2$. For the rest of the proof, we suppose that $\epsilon_0$ is sufficiently small in order to do so.

Along the canonical homotopy, the probability satisfies the following differential equation:
\begin{flalign*}
(\pi +t^2)^3 \twodt p  = & \frac{d^2M}{dt^2} - 3p \pi - 6 \pi t  \frac{dp}{dt} -9 \pi  t^2 p \\
 &  -6 \pi t^3  \frac{dp}{dt} - \frac{15}{4} t^4  p - \frac{3}{2} t^5 \frac{dp}{dt} \\
\end{flalign*}

If $d_{H.}(S,D)< \epsilon_0$, $p(S(t))$ satisfies the following differential inequality for $0 \leq t \leq \|g\|_2$ along the canonical homotopy:

\begin{flalign*}
(\pi +\frac{t^2}{2})^3 \twodt p  \leq & -\frac{L}{2} - 6 \pi t  \frac{dp}{dt} -9 \pi  t^2 p \\
 &  -6 \pi t^3  \frac{dp}{dt} - \frac{15}{4} t^4  p - \frac{3}{2} t^5 \frac{dp}{dt} \\
\end{flalign*}

Furthermore $p(0) = \frac{4}{\pi^2} - \frac18$  and $p^\prime(0) = 0$, as the disk is a critical point for $p$. For $t$ small, consider the barrier function $\bar p(t) = \frac{4}{\pi^2} - \frac18 - \frac{L}{12 \pi^3}t^2$. This function has the same initial conditions as $p$, and satisfies the following:
 
\begin{flalign*}
\left( \pi +\frac{t^2}{2} \right)^3 \twodt \bar p  \geq & -\frac{L}{2} - 6 \pi t  \frac{d \bar p}{dt} -9 \pi  t^2 \bar p \\
 &  -6 \pi t^3  \frac{d \bar p}{dt} - \frac{15}{4} t^4 \bar p - \frac{3}{2} t^5 \frac{d \bar p}{dt} \\
\end{flalign*}

 Therefore, $\bar p$ is a super-solution to the preceding initial value problem, so is an upper bound of $p$. Since $\bar p (t) < \bar p(0)$ whenever $t \neq 0$, this implies that $p(S)< p(D)$.

\end{proof}

 We expect this theorem to hold in $\R^n$ for $n>2$ as well, with a very similar proof. However, the technical lemmas will be incredibly tedious to prove in higher dimensions.

\section{Regions with large isoperimetric ratio have small probability}

We now consider the case when the convex region $S$ is far from being a disk. Heuristically, we expect $p(S)$ to be very small in this case.  If we pick three points from a line segment, they have zero probability of forming a convex triangle. Furthermore, convex regions with large isoperimetric ratio are close to line segments in the Gromov-Hausdorff sense \cite{MG}, so we expect the probability to be small in such a region. This qualitative result is well known in the literature, but we want an effective quantitative version of this, as it may be useful for a full proof of Hall's conjecture. In particular, we prove the following. 

\begin{theorem*}
Let $S$ be a convex subset of $\R^2$ whose isoperimetric ratio is greater than $\frac{7688}{15}$. Then $p(S) < p(D)$. Furthermore, if we denote the isoperimetric ratio of $S$ by $R$, then $p(S) \lesssim R^{-1}$.
\end{theorem*}




\begin{proof}

To begin, we reposition $S$ so as to define its ``height." To do this, consider $L$ the longest segment contained in $S$. We perform a rotation and dilation of $S$ so that $ L= \{ (t,0) ~|~ 0 \leq t \leq 1 \}$. 

This then allows us to define the corresponding ``height" of the region S. To do so, we define the height as the following: \[h(x) = \sup_{(x,t) \in S} t - \inf_{(x,t) \in S} t .\]
We then define the supremum of the heights as $\bar h = \sup_{x \in [0,1]} h(x)$.

 Using the fact $S$ is convex and the definition of the height, it follows that $S$ contains the quadrilateral that has vertices at the origin, $(1,0)$, and the points $(x_0,y_0)$, $(x_0,y_0+\bar h)$ for some $(x_0,y_0) \in \R^2$. This quadrilateral has area at least $\frac{\bar h}{2}$. Similarly, if we consider the functions $g(x) = \sup_{(x,t) \in S} \,t$ and  $f(x) =\inf_{(x,t) \in S} \,t$, we see that both $g(x)$ and $-f(x)$ are convex and bounded above by $\bar h$. Thus, the arc-length of $f$ and $g$ are both bounded by $1+ 2 \bar h$ and so the perimeter is bounded by $2+4 \bar h$. 
  Thus, the isoperimetric ratio of $S$ is at most $\frac{8(1+2 \bar h)^2}{\bar h}$. Hence, if the isoperimetric ratio is very large, $\bar h$ must be very small. This provides an effective estimate that convex regions with large isoperimetric ratio are very long and thin.

In order to translate this estimate into a form that is useful for Hall's problem, we partition $S$ into $N$ subsets in the following way. We let $S_i = \{ (t,y) \in S |~ l_i \leq t \leq l_{i+1} \}$ and pick each cut-off $l_i$ so that that all of the subsets $S_i$ have equal area. Then, we have \[ S = {\bigcup}_{i=1}^N S_i. \]  Strictly speaking, this collection is not a partition of $S$ as the intersection $S_i \cap S_{i+1} \neq \varnothing$, but the intersection of $S_i$ and $S_{i+1}$ has zero measure so can be ignored in this context. Therefore, when we pick three points uniformly at random from $S$, that corresponds to picking three subsets $S_i$ uniformly at random.

 If $N$ is small enough so that each $S_i$ has area greater than $\bar h^2$ ( i.e. $N \leq \frac{1}{2 \bar h}$), then each $S_i$ is wider than it is tall (that is to say, $l_{i+1}-l_i > \bar h$). If we then have three points in distinct non-adjacent regions $S_i$, $S_j$, and $S_k$, it can be seen using Euclidean geometry that the triangle they form is obtuse.
 Therefore, we have that $ 1- p(S)$ is greater than the probability of picking 3 elements uniformly at random from $1, \ldots, N$ such that no two are adjacent or are the same. Using a counting argument, we find that this probability is given by 
\[ \dfrac{ \binom{N-3+1}{3}}{N^3} = \frac{(N-4) (N-3) (N-2)}{N^3}. \] 

 As such, it immediately follows that $ p(S) < 1- \frac{(N-4) (N-3) (N-2)}{N^3}$.


Therefore, if the isoperimetric ratio of $S$ is sufficiently large, we can choose $N$ large enough (while still satisfying $N \leq \frac{1}{2 \bar h}$) so that $\frac{(N-4) (N-3) (N-2)}{N^3}> 9/8-4/\pi^2$. In this case, a random triangle in $S$ is less likely to be acute than a random triangle in the disk. This inequality is satisfied whenever $N \geq 30$. If there are $30$ regions of equal area that are wider than they are tall, then the probability of choosing an acute triangle is smaller than with a disk.

Therefore, if the isoperimetric ratio of $S$ is larger than $7688/15 \approx 512.533$, then we must have $p(S) < p(D)$, so these regions cannot achieve the supremum of the probability. From this, we also see that as the isoperimetric ratio $R$ goes to infinity, the probability of choosing an acute triangle is at most $O(R^{-1})$.

\end{proof}

 The main use of this estimate is to show the supremum of the probability is achieved in a compact set of isoperimetric ratios. Since the only non-compact parts of the moduli space of convex figures (in the $C^0$-norm) correspond to blow-ups of the isoperimetric ratio, the supremum is achieved on a compact part of the moduli space in the $C^0$ topology.

This argument can be generalized somewhat to higher dimensions. We can use similar analysis to show that a unit volume convex region in $\R^3$ with large surface area is close to a two dimensional domain in Hausdorff distance. A priori, however, random triangles in planar regions do not have smaller probability of being acute than random triangles in the ball. In order to show this, the most natural way to do it would be to first prove the two-dimensional Hall's conjecture. 

\section{A potential route to proving Hall's conjecture in two dimensions}

Using the results of Theorem 3 and 4, it is possible to prove Hall's conjecture using a finite calculation. To do this, one uses Theorem 3 to pick $\epsilon_1>0$ and $\delta>0$ so that $\epsilon_1<  d_{H.}(S,D) <2 \epsilon_1$ implies that $p(S) < p(D)- \delta$. One then considers the set $K:= \{ S~ |~ Vol(S)=1, \, d_{H.}(S,D) > 2 \epsilon_1 \textrm{ and }\frac{ L(\partial S)^2}{Vol(S)} < 7688/15 \}$. Since the probability varies continuously, there is an open cover of $K$ in which the probability varies by less than $\delta/4$. Since $K$ is compact, there is a finite subcover $U_i$. Therefore it suffices to check that $p(x_i) < p(S) - \delta$ where $x_i$ is an arbitrary convex region in $U_i$. 

 
Although there are only a finite number of cases to check, the fact that the space of convex regions is infinite dimensional makes the number of cases extremely large. It is possible to put some crude bounds on the number of open sets of radius $d$ needed to cover the set of unit volume convex regions with a given bound on the isoperimetric ratio. It is possible to derive a naive estimate of $O(C^{1/d^2})$, with $C$ growing exponentially in the bounds of the isoperimetric ratio. Using the convexity condition carefully we believe that this can be sharpened to $O(C^{1/d})$, but this is still very fast growing. 
 Our Lipschitz estimates on $p$ can be used to give a lower bound for $d$, but this will be quite small. With the current estimates, the number of calculations necessary is extremely large and so the computation is completely intractable. However, there is a fair amount of optimization that can be done. We should also note that computing $p(S)$ for any region is very computationally difficult. If we can estimate the probability effectively, then this would be helpful in making this program realistic.

  \section{Acknowledgements and some resources}

This work was partially supported by DARPA/ARO Grant W911NF-16-1-0383 (PI: Jun Zhang, University of Michigan).

While working on this problem, I relied on the help of many different people. I am especially thankful to Kori Khan for her support and contributions. In particular, she observed that the distribution of the angle between two points on the circle could be expressed as an autocorrelation integral. This insight was crucial to finding the right approach for the problem.

I am grateful to the following people; Bob Stanton and Mike Belfanti for their help with the representation theoretic elements of the proof; Tom Needham for his helpful suggestions; Mizan Khan for help editing this paper.

 On a personal note, Glen Hall was my advisor through my undergraduate studies and I am extremely grateful for his mentorship. He first introduced me to this problem as a suggested project for an undergraduate thesis. I was unable to make any progress at the time, but over the years I regularly revisited the problem. 
 
 While working on this project, two extremely useful resources were the papers of Eisenberg and Sullivan ~\cite{ES}~\cite{ES2}. The former gives a good background on random triangles, with some very useful computations. The latter gives an excellent exposition about Crofton's equation and its generalizations and was extremely valuable to my understanding of the problem. Another resource that was indispensable was Geogebra, which could be used to generate manipulable figures to help understand the geometry. For the interested reader, I have created some notebooks in GeoGebra that might be helpful in visualizing the terms in the calculation and would be willing to provide them.

 \bibliography{Hallbib.bib}
\bibliographystyle{alpha}

\appendix

\section{}

Until now, we have tried to avoid too many technical details so as to focus on the overall structure of the proof. However, for completeness we must prove the lemmas used in our theorems. For convenience, we will restate each of the lemmas at the beginning of each section. For the lemmas in Theorem 3, we found Geogebra to be extremely useful for visualizing the geometry. We strongly recommend using this or some similar software to help gain familiarity with the approach, especially for the final parts. With a good picture, it is not too difficult to see what terms should be controlled and what general strategies should be used. The Geogebra files that we used are available online \cite{Website}.

\section{The proof of the lemmas used in Theorem 2}

\begin{lemma*}[Properties of the Fourier series of an autocorrelation]
 
Suppose that $\mu: \mathbb{S}^2 \to \mathbb{R}$ is a smooth function which decomposes as $\mu = \sum_{i,m} c_i^m Y_m^i$ where $Y_m^i$ are the spherical harmonics.

Consider the autocorrelation $R_\mu$ and write it in terms of spherical harmonics. That is to say, $R_\mu = \sum_{i,m} b_i^m Y_m^i$. Then the coefficients $b_i^m$ have the following properties.
\begin{enumerate}
\item For all $i, m$, $b_i^m \geq 0$. Furthermore, $b_i^m=0$ for all $i \neq 0.$
\item For any $i$ and $m$, of $c_i^m \neq 0$ then $b_0^m > 0$.
\item The $c_i^m$ terms do not contribute to $b_j^n$ for $n \neq m$. In particular, none of the higher $c_i^m$ terms contribute to $b_0^0$ or $b_0^1$, which are the only positive terms in the Fourier series of $A_3(\theta)$.
\end{enumerate}

\end{lemma*}
 
 \begin{proof}

 We will give some indication of how to prove this lemma. The first two claims follow from Fourier analysis on $SO(3)$  (see ~\cite{RT}). On $SO(3)$, there is a correlation theorem, which follows from the convolution theorem. It states the following, where the overline denotes conjugation;
 
 \[ \mathcal{F}[R_\mu]= \overline{ \mathcal{F}[\mu]} \times \mathcal{F}[\mu] \]
 
 For our purposes, this immediately shows that $b_0^m \geq 0$ and that if $c_i^m \neq 0$ for any $i$,  $b_0^m >0$, as desired. 
 As in the two-dimensional case, it is worth noting that when we write $R_\mu = \sum_{i,m} b_i^m Y_m^i$ from $\mu = \sum_{i,m} c_i^m Y_m^i$, the terms $c_i^m$ with $i \neq 0$ contribute to $b_0^m$. 
 
For the third claim, we must show that the $c_i^n$ contribute to $b_0^n$ but not to $b_0^m$ for $m \neq n$.
There is a simple way to show this without calculating tedious integrals or diving into the correlation theorem. Instead, we can apply the Laplacian to the following integral expression:
 
 \begin{eqnarray*}
 \Delta_g \int_{x \in \mathbb{S}^2} Y_m^i(x) Y_n^j(xg) dS & = &  \int_{x \in \mathbb{S}^2} Y_m^i(x) \Delta_g Y_n^j(xg) dS  \\
									 & = &  \int_{x \in \mathbb{S}^2} Y_m^i(x) n(n+1) Y_n^j(xg) dS  \\
									 & = & n(n+1)  \int_{x \in \mathbb{S}^2} Y_m^i(x) Y_n^j(xg) dS  \\
  \end{eqnarray*}
  
  Using symmetry and a change of variables, this must be zero if $m \neq n$. From this, it follows that the $c_i^m$ terms do not contribute to $b_j^n$ for $n \neq m$ (as when we decompose $\mu$ in spherical harmonics, only the matching terms contribute to the Fourier series of the autocorrelation).  \end{proof}

\subsection{The proof of the lemma on Legendre polynomials}

Here we prove the following lemma which we used in Theorem 2. A Mathematica notebook with the calculations can be found online \cite{Website}.
 
 \begin{lemma*}
 
 Let $P_n(\cos (\theta ))$ be the $n$-th Legendre polynomial evaluated at $\cos (\theta )$. Then the following inequalities hold:
 
   \begin{eqnarray*} \int_0^\pi A_3(\theta) P_n(\cos (\theta )) \sin{\theta} \,d \theta
  \begin{array}{ll}
>0 \textrm{ for } n = 0,1  \\
  \leq 0 \textrm{ otherwise}  \\
\end{array} 
   \end{eqnarray*}

\end{lemma*}

\begin{proof}

 Using an explicit formula for $P_n(\cos (\theta ))$, this reduces to the following: 
 
\begin{eqnarray*} \int_0^\pi  \sum_{k=0}^n {n\choose k} {-n-1\choose k} \left( \sin(\theta/2) \right)^k A_3(\theta) \sin{\theta}\, d \theta >0 \textrm{ when n = 0,1} \\
\int_0^\pi  \sum_{k=0}^n {n\choose k} {-n-1\choose k} \left( \sin(\theta/2) \right)^k A_3(\theta) \sin{\theta}\, d \theta <0 \textrm{ otherwise} \\
 \end{eqnarray*}
 
For the three dimensional ball, we can simplify \ref{A 3D} to obtain the following expression for $A_3(\theta) $:

\begin{eqnarray*}
A_3(\theta) &= &- \frac { 2 } { 3 } \pi \left( 1 + \sin \left[ \frac { \theta } { 2 } \right] ^ { 3 } \right) \\
&   &+ \pi \left( \frac { 1 } { 12 } \left( 9 \cos \left[ \frac { \theta } { 2 } \right] - \cos \left[ \frac { 3 \theta } { 2 } \right] \right) + \frac { 1 } { 6 } \left( 9 \sin \left[ \frac { \theta } { 2 } \right] + \sin \left[ \frac { 3 } { 2 } \right] \right) \right)
 \end{eqnarray*}

Thus, we have the following expression for $A_3(\theta) \sin(\theta)$:
\begin{flalign*}
&A_3(\theta) \sin(\theta)  \\
& = \frac { 1 } { 24 } \left( 32 \pi \cos \left[ \frac { \theta } { 2 } \right] - 24 \pi \cos \left[ \frac { 3 \theta } { 2 } \right] - 8 \pi \cos \left[ \frac { 5 \theta } { 2 } \right] + 10 \sin \left[ \frac { \theta } { 2 } \right] + 64 \pi \sin [ \theta ] + 9 \sin \left[ \frac { 3 \theta } { 2 } \right] - \sin \left[ \frac { 5 \theta } { 2 } \right] \right)
 \end{flalign*}

For $n \geq 1$, we now want to compute the following: 

\begin{equation} \label{Fourier A 3D}
\int_0^\pi  \sum_{k=0}^n {n\choose k} {-n-1\choose k} \left( \sin(\theta/2) \right)^k A_3(\theta) \sin{\theta}\, d \theta
\end{equation}

In practice, we found trying to do this integral directly with Mathematica caused the computation to hang up. To avoid this problem, we derived the following intermediate identity so that the computation was tractable:
\begin{flalign*}
& \int_0^{\pi } \frac{1}{24} \left(
\begin{array}{ll}
10 \pi  \sin \left(\frac{\theta }{2}\right)+9 \pi  \sin \left(\frac{3 \theta }{2}\right)-\pi  \sin \left(\frac{5 \theta }{2}\right)+64 \pi  \sin (\theta )  \\
+ 32 \pi  \cos \left(\frac{\theta }{2}\right)-24 \pi  \cos \left(\frac{3 \theta }{2}\right)-8 \pi  \cos \left(\frac{5 \theta }{2}\right)  \\
\end{array} 
\right) \sin ^{2 k}\left(\frac{\theta }{2}\right) \, d\theta   \\
& =\frac { ( 4 + 8 \mathrm { k } ) \pi } { 15 + 31 \mathrm { k } + 20 \mathrm { k } ^ { 2 } + 4 \mathrm { k } ^ { 3 } } + \frac { ( 2 + \mathrm { k } ) \pi ^ { 3 / 2 } \Gamma  [ 1 + \mathrm { k } ] } { \Gamma  \left[ \frac { 7 } { 2 } + \mathrm { k } \right] } 
 \end{flalign*}

Substituting this identity into the previous one (\ref{Fourier A 3D}) and switching the order of the summation and integration using Fubini's theorem, we find the following:

\begin{flalign*}
&\int_0^\pi  \sum_{k=0}^n {n\choose k} {-n-1\choose k} \left( \sin(\theta/2) \right)^k A_3(\theta) \sin{\theta}\, d \theta
\end{flalign*}
\begin{flalign} 
& = - \frac { 8 \left( - 9 - 6 ( - 1 ) ^ { m } + 4 m + 2 ( - 1 ) ^ { m } m + 4 m ^ { 2 } + 2 ( - 1 ) ^ { m } m ^ { 2 } \right) \pi } { ( - 3 + 2 m ) ( - 1 + 2 m ) ( 1 + 2 m ) ( 3 + 2 m ) ( 5 + 2 m ) } 
\end{flalign}

We note that (9) is positive for $n=1$ and negative otherwise. For $n=0$, we find the following:

 $$\int_0^\pi A_3(\theta) \sin{\theta} \,d \theta = \frac{4 \pi }{3} >0$$

Note that (9) is not equal to $\int_0^\pi A_3(\theta) P_0(\cos (\theta )) \sin{\theta} \,d \theta $ when $n=0$ due to a division by zero issue.

\end{proof}

\section{Existence of the canonical homotopy (Lemma 10)}


\begin{lemma*}
Let $S$ be convex region whose Hausdorff distance from the disk is no more than $\epsilon$, with $\epsilon < \frac{1}{8\sqrt{2 \pi}}$.
There is an embedding $\bar S  = \{ re^{i \theta}~ |~ r \leq 1 + \bar g(\theta) \}$ with the following properties.

\begin{enumerate}
\item $\bar S$ is similar to $S$. 
\item $\bar S$ can be obtained by translating $S$ by no more than $3\epsilon$ and dilating $S$ by a factor between $1-3\epsilon$ and $1+3\epsilon$. In other words, $d_{H.}(S, \bar S) < 7 \epsilon$.
\item The function $\bar g(\theta)$ satisfies the following:
 $$ \int_0^{ 2 \pi} \bar g(\theta) d \theta = 0 \hspace{2cm} \int_0^{ 2 \pi} \bar g(\theta) e^{\ i \theta} d \theta = 0$$
\end{enumerate}
\end{lemma*}

\begin{proof}
We want to find a similarity of $S$ that satisfies the third condition, and show that it satisfies the second condition. Unfortunately, it is impossible to directly compute the integrals in the third condition for arbitrary $S$. Therefore, we need to find a work-around. 

Heuristically, we expect that dilations should change the total integral of $\bar g(\theta)$ while leaving its first Fourier coefficients roughly constant. Similarly, we expect that translations do not change the total integral of $\bar g(\theta)$ much while changing the first Fourier coefficients. Unfortunately, this ansatz is not exactly true and so to salvage this intuition, we need to work topologically.

 Before proving the estimates, we'll explain the general strategy in more detail. Given a vector $w \in \R^2$ and a dilation factor $s \in R$, we can consider the map $T$:
 
 $$T: \R^2 \times \R \to \mathbb{C} \times \R$$  
$$T(w,s) = \left( \int_0^{ 2 \pi} g_{sS+w}(\theta) e^{ i \theta} \, d \theta , \int_0^{ 2 \pi} g_{sS+w}(\theta) \, d \theta \right)$$
 
We want to find a vector $w_0$ and a dilation $s_0$ with $|w_0|, |s_0-1|< 3 \epsilon$ for which $T$ vanishes. To do so, assume that this is not the case. Then since $T$ does not vanish on $B_{3\epsilon}(0) \times [1-3\epsilon, 1 + 3\epsilon]$, the following map is well defined and continuous:

\[  \bar T: B_{3\epsilon}(0) \times [1-3\epsilon, 1 + 3\epsilon] \to \mathbb{S}^2 \]
\[ \bar T(w,s) = \frac{T(w,s)}{|T(w,s)|} \]
 
 We can restrict this map to the boundary of $B_{3\epsilon}(0) \times [1-3\epsilon, 1 + 3\epsilon]$. Topologically, this restriction is a continuous self-map of the sphere. The estimates here establish that the degree of this map is $1$, after which we can use the degree theorem to obtain a contradiction. 

 \subsection{The estimates in Lemma 10}
 
 Since $S$ is Hausdorff distance at most $\epsilon$ from $D$, $S$ is completely contained within the disk of radius $1+\epsilon$ and completely contains the disk of radius $1-\epsilon$ . Therefore, the dilation of $S$ by a factor of $(1+3 \epsilon)$ is completely contained in the disk of radius $1+ 4 \epsilon +3\epsilon^2$ and completely contains the disk of radius $(1+2\epsilon-3\epsilon^2)$.

 Suppose we dilate $S$ by a factor of $(1+3 \epsilon)$ and translate the $S$ by some amount $3 \epsilon_T$ in a unit direction $V$, where $\epsilon_T$ is no greater than $\epsilon$. Then we can use the inner disk to obtain a lower estimate on $\int_0^{ 2 \pi} \bar g(\theta) d \theta$.
 
 \begin{eqnarray*}
  \int_0^{ 2 \pi} \bar g_{(1+3\epsilon)S+3\epsilon_T V}(\theta) d \theta & \geq &   \int_0^{ 2 \pi} 2 \epsilon_T \cos(\theta) + \sqrt{(1+2\epsilon-3\epsilon^2)^2-9 \epsilon^2_T \sin^2(\theta) }-1 \,d \theta \\
  &=&  \int_0^{ 2 \pi}  \sqrt{ 1+4\epsilon-4\epsilon^2 - 12 \epsilon^3+9 \epsilon^4 - 9 \epsilon^2_T  \sin^2(\theta)} -1 \,d \theta \\
  &\geq & \int_0^{ 2 \pi} 2\epsilon -O(\epsilon^2) - O(\epsilon^2_T) \, d \theta \\
  & = & 4 \pi \epsilon - O(\epsilon^2) 
 \end{eqnarray*}
 
We can use the exact same argument to show that if we contract the disk by $1-3\epsilon$, and translate it by no more than $3 \epsilon$ in any direction, then $  \int_0^{ 2 \pi} \bar g_{(1-3\epsilon)S+3\epsilon_T V}(\theta) d \theta < - 4 \pi \epsilon+ O(\epsilon^2) $.

 Suppose now that we translate $S$ by $3 \epsilon$ in some direction $V$, after dilating it by $1+3\epsilon_D$, with $\epsilon_D<\epsilon$. Using the inner disk, we obtain the following estimate:
 
  \begin{eqnarray*}
  \int_0^{ 2 \pi} \bar g_{(1+3\epsilon_D)S+2\epsilon V}(\theta) d \theta & \geq &   \int_0^{ 2 \pi} 2 \epsilon \cos(\theta) + \sqrt{((1-\epsilon)^2(1+3 \epsilon_D)^2-4 \epsilon^2 \sin^2(\theta) }-1 \,d \theta \\
  & \geq&  \int_0^{ 2 \pi}  \sqrt{ 1-2\epsilon+ 6 \epsilon_D - O(\epsilon^2)} -1 \,d \theta \\
  &\geq & \int_0^{ 2 \pi} -\epsilon + 3 \epsilon_D -O(\epsilon^2) \, d \theta \\
  & = & -2 \pi \epsilon + 6 \pi \epsilon_D + O(\epsilon^2) 
 \end{eqnarray*}

Similarly, we want to estimate $\int_0^{ 2 \pi} \bar g(\theta) e^{i \theta} d \theta$ using the inner and outer disks. After performing a rotation, we can assume that $V$ is the unit vector in the $x$-direction. For conciseness we only consider the cosine term here. Suppose we translate S by $3 \epsilon$ in some direction, after dilating S by $1+3\epsilon_D$, with $|\epsilon_D| < \epsilon$. We can use the inner and outer disks to make the following estimates:

 \begin{flalign*}
  \int_0^{ 2 \pi} \bar g_{(1+3\epsilon_D)S+3\epsilon V}(\theta) \cos(\theta) d \theta & \geq    \int_{-\pi/2}^{ \pi/2} \left(3 \epsilon \cos(\theta) + \sqrt{(1-\epsilon)^2(1+3 \epsilon_D)^2-9 \epsilon^2 \sin^2(\theta) }-1 \right) \cos(\theta) \,d \theta \\
& +   \int_{\pi/2}^{ 3\pi/2} \left( 3 \epsilon \cos(\theta) + \sqrt{(1+\epsilon)^2(1+3 \epsilon_D)^2-9 \epsilon^2 \sin^2(\theta) }-1 \right) \cos(\theta) \,d \theta \\
  &\geq  3 \pi \epsilon + \int_{-\pi/2}^{ \pi/2}  \left( - \epsilon + 3 \epsilon_D - O(\epsilon^2) \right) \cos(\theta) \,d \theta \\
  &  + \int_{\pi/2}^{3 \pi/2}  \left(  \epsilon + 3 \epsilon_D + O(\epsilon^2) \right) \cos(\theta) \,d \theta \\
  & =  3 \pi \epsilon + 2 \left( - \epsilon + 3 \epsilon_D  \right)  -  2 \left(  \epsilon + 3 \epsilon_D  \right) - O(\epsilon^2) \\
    & =  3 \pi \epsilon  - 4 \epsilon - O(\epsilon^2)
 \end{flalign*}
 
 Note that we did not assume a sign for $\epsilon_D$, so this covers both the cases when the dilation expands or contracts $S$.
 Finally, we suppose now that we dilate $S$ by $1+3\epsilon$ and translate it by $3 \epsilon_T$ in some direction $V$ with $\epsilon_T<\epsilon$. Using the inner and outer disks, we obtain the following estimate. Note that we would obtain the same estimate if we dilated $S$ by $1-3\epsilon$ instead.
 
  \begin{flalign*}
  \int_0^{ 2 \pi} \bar g_{(1+3\epsilon)S+3\epsilon_T V}(\theta) \cos(\theta) d \theta  & \geq    \int_{-\pi/2}^{ \pi/2} \left(3 \epsilon_T \cos(\theta) + \sqrt{(1-\epsilon)^2(1+3 \epsilon)^2-9 \epsilon_T^2 \sin^2(\theta) }-1 \right) \cos(\theta) \,d \theta \\
& +   \int_{\pi/2}^{ 3\pi/2} \left( 3 \epsilon_T \cos(\theta) + \sqrt{(1+\epsilon)^2(1+3 \epsilon)^2-9 \epsilon_T^2 \sin^2(\theta) }-1 \right) \cos(\theta) \,d \theta \\
  &\geq  3 \pi \epsilon_T + \int_{-\pi/2}^{ \pi/2}  \left(2 \epsilon - O(\epsilon^2) \right) \cos(\theta) \,d \theta \\
  &  + \int_{\pi/2}^{3 \pi/2}  \left( 4 \epsilon - O(\epsilon^2) \right) \cos(\theta) \,d \theta \\
  & =  3 \pi \epsilon_T + 2 \left( 2 \epsilon  \right)  -  2 \left(  4 \epsilon  \right) - O(\epsilon^2) \\
    & =  3 \pi \epsilon_T  - 4 \epsilon - O(\epsilon^2)
 \end{flalign*}

Now that we have proven the four estimates needed, we can finish the proof of Lemma 13. To prove that an embedding of $S$ satisfying the previous conditions exists, we consider the continuous map $T$.

$$T:B_{3\epsilon}(0) \times [1-3\epsilon, 1 + 3\epsilon] \to \mathbb{R}^3$$  
$$T(w,s) = \left( \int_0^{ 2 \pi} g_{sS+w}(\theta) e^{ i \theta} \, d \theta , \int_0^{ 2 \pi} g_{sS+w}(\theta) \, d \theta \right)$$

We want to show that there is a pair $(w,s)$ so that $T$ vanishes. Suppose, for the sake of contradiction, that $T$ never vanishes. If so, it induces a continuous map: $$\bar T: B_{3 \epsilon}(0) \times [1-3\epsilon, 1 + 3\epsilon] \to \mathbb{S}^2$$  $$\bar T(w,s) = \frac{T(w,s)}{|T(w,s)|}.$$ 

The restriction of this map to the boundary (which we denote $\bar T|_{\mathbb{S}^2}$) is topologically a self-map of the sphere. We want to compute the degree of this map. To do so, let $ x= (3 \epsilon_T V,3 \epsilon) \in B_{3 \epsilon}(0) \times [1-3\epsilon, 1 + 3\epsilon]$, and observe we have the following estimate:

 \begin{eqnarray*}
\langle  \bar T|_{\mathbb{S}^2} (x), x\rangle & \geq &  3 \epsilon_T \cdot \int_0^{ 2 \pi} \bar g_{(1+3\epsilon_D)S+3\epsilon V}(\theta) \cos(\theta) d \theta +  3 \epsilon  \int_0^{ 2 \pi} \bar g_{(1+3\epsilon)S+3\epsilon_T V}(\theta) d \theta \\
& \geq & 3 \epsilon_T( 3 \pi \epsilon_T  - 4 \epsilon - O(\epsilon^2)) + 3 \epsilon (4 \pi \epsilon - O(\epsilon^2)) \\
& \geq & (12 \pi - 12) \epsilon^2 - O(\epsilon^3)
  \end{eqnarray*}
 
 If $\epsilon$ sufficiently small, this is positive. Repeating this argument for the other parts of the boundary of $B_{3\epsilon}(0) \times [1-3\epsilon, 1 + 3\epsilon]$, we can show that $\bar T|_{\mathbb{S}^2}$ satisfies $\langle x, \bar T|_{\mathbb{S}^2}( x) \rangle > 0$ for all $x$. As a result, $\bar T|_{\mathbb{S}^2}(x)$ is contained in the same hemisphere as $x$. This implies the degree of the $T$ is one, as we can construct a homotopy from $\bar T|_{\mathbb{S}^2}$ to the identity map. Therefore, the supposed map $\bar T$ retracts $\bar T|_{\mathbb{S}^2}$, which is impossible. As a result, $T$ must vanish for some translation $w_0$ and some dilation $s_0$.

To finish the proof, we observe that since $|w_0| < 3 \epsilon$ and $|s_0|<3\epsilon$, we have that $d_{H.}(S, \bar S)< 3 \epsilon(1+3 \epsilon)+3 \epsilon$. Using the assumption that $\epsilon < \frac{1}{8 \sqrt{2 \pi}}$, this implies that $d_{H.}(S, \bar S) < 7 \epsilon$.

\end{proof}

In the previous calculations, it should be noted that the terms that were $O(\epsilon_D)$ and $O(\epsilon_T)$ cancelled out, which shows that the effect of translations on the total integral of $g(\theta)$ is $0$ to first order and that the effect of dilations on the first Fourier coefficients of $g(\theta)$ is also $0$ to first order. In other words, our initial ansatz was correct to first order. We'll also note that these estimates are rather finicky, and the factor of $3 \epsilon$ for the translations and dilations cannot be reduced much.


\section{The Lipschitz estimate of $p(S(t))$ along the canonical homotopy}

 \noindent {\bf Notation}: Before starting the proof, we introduce the following notation, which will help to write certain sets concisely.
  We will use $S \Delta D$ to denote the symmetric difference between the sets $S$ and $D$:
\[ S \Delta D= (S \backslash D) \cup (D \backslash S). \]

In this section, we prove the following lemma.

\begin{lemma*}[The Lipschitz estimate on $p(S)$] 
There exist uniform constants $ e, C > 0$ so that whenever $\epsilon < e$ and $d_{H.}(S,D) \leq \epsilon$, then $ \bigl\lvert \bigl\lvert p(S) - p(D) \bigl\rvert \bigl\rvert < C \epsilon $.
\end{lemma*}

\begin{proof}

To prove this lemma, we suppose that $d_{H.}(S, D) < 8 \epsilon$ and assume that $\epsilon < \frac{1}{8 \sqrt{2 \pi}}$. This assumption implies that $\| g \|_{L^2} < 1$, which we use in the first inequality. The third inequality estimates the size of the set $(S \times S \times S) \Delta (D \times D \times D)$. The fourth inequality uses Taylor's theorem to control the first term.

\begin{flalign*}
| p(S) - p(D) |  = & \Bigg | \frac{\int_S \int_S \int_S f(x,y,z) \,  dx \,  dy \,  dz}{V(S)^3} - \frac{\int_D \int_D \int_D f(x,y,z) \,  dx \,  dy \,  dz}{V(D)^3} \Bigg| \\
   = & \Bigg | \frac{ \pi^3 \int_S \int_S \int_S f(x,y,z) \,  dx \,  dy \,  dz  - (\pi+ \frac{1}{2} \| g(\theta) \|_2^2)^3 \int_D \int_D \int_D f(x,y,z) \,  dx \,  dy \,  dz}{ \pi^3(\pi+ \frac{1}{2} \| g(\theta) \|_2^2)^3} \Bigg | \\
   \leq &  \pi^3 \Bigg | \frac{ \int_S \int_S \int_S f(x,y,z) \,  dx \,  dy \,  dz  - \int_D \int_D \int_D f(x,y,z) \,  dx \,  dy \,  dz}{ \pi^6} \Bigg | \\
   & + \frac{7}{2}  \| g(\theta) \|_2^2 \Bigg | \frac{  \int_D \int_D \int_D f(x,y,z) \,  dx \,  dy \,  dz}{\pi^6} \Bigg | \\
    \leq &  \frac{ \int_{(S \times S \times S) \Delta (D \times D \times D)} 1 \,  dx \,  dy \,  dz }{ \pi^3} + \frac{7}{2}  \| g(\theta) \|_2^2  \frac{  \frac{4}{\pi^2} - \frac18}{\pi^3} \\
       \leq &~ \frac{3}{\pi^3} \left( \pi (1 + 8 \epsilon)^2 \right)^2  \pi \left( (1 + 8 \epsilon)^2-(1 - 8 \epsilon)^2\right)   + \frac{7}{2}  \| g(\theta) \|_2^2  \frac{  \frac{4}{\pi^2} - \frac18}{\pi^3} \\
          \leq &~ 6(1 + 8 \epsilon)^5 8 \epsilon  + 7 \pi  (8 \epsilon)^2 \frac{  \frac{4}{\pi^2} - \frac18}{\pi^3} \leq 48 \epsilon+O(\epsilon^2)
\end{flalign*}

As this upper bound is $O(\epsilon)$, the desired estimate holds.

\end{proof}

\section{The H\"older estimate on $ \frac{d^2}{dt^2} M$}

Here, we start the proof of the H\"older estimate by reducing it to three more manageable lemmas which we prove in the following subsections.
 Although this part of the proof does not require any abstract machinery, we have to carefully estimate each term using geometry. 
 
  Recall that we are proving the following lemma.

\begin{lemma*}[The H\"older-$1/2$ estimate on  on $\frac{d^2M}{dt^2}$]
There exist uniform constants $ e, C >0$ so that whenever $\epsilon < e$ and $d_{H.}(S,D) \leq \epsilon$, then $ \bigl\lvert \frac{d^2}{dt^2} M (S) - \frac{d^2}{dt^2} M (D) \bigl\rvert < C \epsilon^{1/2} $.
\end{lemma*}

  The H\"older $1/2$ estimate is natural, but it is likely possible to strengthen it with more effort. It seems that a full $C^3$ estimate does not hold for general convex regions near the disk, so the $C^{2,\alpha}$ estimate is likely optimal. The general strategy to prove this lemma is to expand $\frac{d^2}{dt^2} M (S)$ explicitly and to carefully control each term. 

  \begin{proof}
  
   Recall that the second variation of $M$ is the following:

\begin{flalign*}
\twodt M = &  3 \int_0^{2 \pi}  \int_S  \int_S   f(r( \theta_1, t) e^{i \theta_1},y,z) \dt  \mu(\theta_1, t) r(\theta_1, t) \,dy \, dz \, d \theta_1  \\
&+  6 \int_0^{2 \pi} \int_0^{2 \pi} \int_S  f(r(\theta_1, t) e^{i \theta_1},r(\theta_2, t) e^{i \theta_2},z) \mu(\theta_1, t) \mu(\theta_2, t) r(\theta_1, t) r(\theta_2, t)  \, dz \,  d \theta_1 \, d \theta_2  \\
&+  3 \int_0^{2 \pi} \int_S \int_S  \frac{\partial}{\partial r} f(r(\theta_1, t) e^{i \theta_1},y,z) \mu(\theta_1)^2 r(\theta_1, t) \, dy \,d z \, d \theta_1  \\
&+  3 \int_0^{2 \pi} \int_S \int_S  f(r(\theta_1, t) e^{i \theta_1},y,z) \mu(\theta_1)^2 \, dy \,d z \, d \theta_1 
\end{flalign*}

Since $\mu(\theta,t)$ is time-independent along the canonical homotopy, we write it as $\mu(\theta)$ and observe that the first term vanishes.

\begin{flalign*}
 \twodt M  = &  6 \int_0^{2 \pi} \int_0^{2 \pi} \int_S  f(r(\theta_1, t) e^{i \theta_1},r(\theta_2, t) e^{i \theta_2},z) \mu(\theta_1) \mu(\theta_2) r(\theta_1, t) r(\theta_2, t)  \, dz \,  d \theta_1 \, d \theta_2  \\
&+  3 \int_0^{2 \pi} \int_S \int_S  \frac{\partial}{\partial r} f(r(\theta_1, t) e^{i \theta_1},y,z) \mu(\theta_1)^2 r(\theta_1, t) \, dy \,d z \, d \theta_1  \\
&+  3 \int_0^{2 \pi} \int_S \int_S  f(r(\theta_1, t) e^{i \theta_1},y,z) \mu(\theta_1)^2 \, dy \,d z \, d \theta_1 
\end{flalign*}

The rest of the proof is showing that each of these terms are uniformly H\"older continuous in $d_{H.}(D,S)$. So as to be more manageable, we split this into three lemmas for each of these terms.

\begin{lemma}[H\"older continuity of the autocorrelation terms]
There exist uniform constants $ e ,C >0$ so that whenever $d_{H.}(S,D) \leq \epsilon < e$, then 

\begin{eqnarray*} \Bigg | \int_0^{2 \pi} \int_0^{2 \pi} \int_S   f(r(\theta_1,t) e^{i \theta_1},r(\theta_2,t) e^{i \theta_2},z) \mu(\theta_1) \mu(\theta_2) r(\theta_1, t) r(\theta_2, t) \, dz \,  d \theta_1 \, d \theta_2  & \\
-   \int_0^{2 \pi} \int_0^{2 \pi} \int_D   f(e^{i \theta_1}, e^{i \theta_2},z) \mu(\theta_1) \mu(\theta_2) \, dz \,  d \theta_1 \, d \theta_2 \Bigg |  & < C \epsilon^{1/2} 
\end{eqnarray*}
\end{lemma}

\begin{lemma}[Holder continuity of the terms involving $\frac{\partial f}{\partial r}$]
There exist uniform constants $ e, C >0$ so that whenever $d_{H.}(S,D) \leq \epsilon < e$, then 

\begin{eqnarray*} \Bigg |  \int_0^{2 \pi} \int_S \int_S  \frac{\partial}{\partial r} f(r(\theta_1,t)e^{i \theta_1},y,z) r(\theta_1,t) (\mu(\theta_1))^2 \, dy \,d z \, d \theta & \\
-    \int_0^{2 \pi} \int_D \int_D  \frac{\partial}{\partial r} f(e^{i \theta_1},y,z) (\mu(\theta_1))^2 \, dy \,d z \, d \theta_1 \Bigg |  & < C \epsilon^{1/2} 
\end{eqnarray*}

\end{lemma}

\begin{lemma} [Lipschitz continuity of the arc-length term]
There exist uniform constants $ e, C >0$ so that whenever $\epsilon < e$ and $d_{H.}(S,D) \leq \epsilon$, 
\begin{eqnarray*}
 \Bigg | \int_0^{2 \pi} \int_S \int_S  f(r(\theta_1, t) e^{i \theta_1},y,z) \mu(\theta_1)^2 \, dy \,d z \, d \theta_1 -
  \int_0^{2 \pi} \int_D \int_D  f(e^{i \theta_1},y,z) \mu(\theta_1)^2 \, dy \,d z \, d \theta_1 \Bigg | < C \epsilon
  \end{eqnarray*}
\end{lemma}

\subsection{H\"older continuity of the autocorrelation term}

In this section, we prove Lemma 14, which we restate here.

\begin{lemma*}
There exist uniform constants $ e ,C >0$ so that whenever $d_{H.}(S,D) \leq \epsilon < e$, then 

\begin{eqnarray*} \Bigg | \int_0^{2 \pi} \int_0^{2 \pi} \int_S   f(r(\theta_1) e^{i \theta_1},r(\theta_2) e^{i \theta_2},z) \mu(\theta_1) \mu(\theta_2) r(\theta_1, t) r(\theta_2, t) \, dz \,  d \theta_1 \, d \theta_2  & \\
-   \int_0^{2 \pi} \int_0^{2 \pi} \int_D   f(e^{i \theta_1}, e^{i \theta_2},z) \mu(\theta_1) \mu(\theta_2) \, dz \,  d \theta_1 \, d \theta_2 \Bigg |  & < C \epsilon^{1/2} 
\end{eqnarray*}
\end{lemma*}

\begin{proof}

To prove this, we note that it is sufficient to show that there is a $C$, uniform in $\theta_1$ and $\theta_2$, so that the following estimate holds:
\[ \Bigg | \int_S   f(r(\theta_1,t ) e^{i \theta_1},r(\theta_2,t) e^{i \theta_2},z) r(\theta_1, t) r(\theta_2, t) \, dz - \int_D   f(e^{i \theta_1}, e^{i \theta_2},z) \, dz \Bigg | < C \epsilon^{1/2}
\]

If we can show this estimate, it immediately implies the following
\begin{eqnarray*} \Bigg | \int_0^{2 \pi} \int_0^{2 \pi} \int_S   f(r(\theta_1,t) e^{i \theta_1},r(\theta_2,t) e^{i \theta_2},z) \mu(\theta_1) \mu(\theta_2) r(\theta_1, t) r(\theta_2, t) \, dz \,  d \theta_1 \, d \theta_2  & \\
-   \int_0^{2 \pi} \int_0^{2 \pi} \int_D   f(e^{i \theta_1}, e^{i \theta_2},z) \mu(\theta_1) \mu(\theta_2) \, dz \,  d \theta_1 \, d \theta_2 \Bigg | & \\
 < \int_0^{2 \pi}  \int_0^{2 \pi} C \epsilon^{1/2} \,  d \theta_1 \, d \theta_2 = 4 \pi^2 C \epsilon^{1/2} &\\
\end{eqnarray*}

To do this, it is convenient to let $X=r(\theta_1) e^{i \theta_1}$, $Y=r(\theta_2) e^{i \theta_2}$, and use the triangle inequality to estimate the previous term.

\begin{eqnarray*}
\Bigg | \int_S   f(X,Y,z) |X| |Y| \, dz - \int_D   f\left( \frac{X}{|X|}, \frac{Y}{|Y|},z \right) \, dz \Bigg | &\leq& \Bigg | \int_S   f(X,Y,z) \, dz - \int_D   f(X,Y,z) |X| |Y| \, dz \Bigg | \\
 & & +\Bigg | \int_D   f(X,Y,z) |X| |Y| \, dz - \int_D   f\left( X,Y,z \right) \, dz \Bigg | \\
 & & +\Bigg | \int_D   f(X,Y,z) \, dz - \int_D   f\left( \frac{X}{|X|}, \frac{Y}{|Y|},z \right) \, dz \Bigg | 
\end{eqnarray*}

To prove this lemma, we show that all of these terms are $O(\epsilon^{1/2})$. 

For the first term, note that $f(X,Y,z) \leq 1$ and that $d_{H.}(S,D) < 8 \epsilon$. Therefore, we have the following estimate:

$$\Bigg | \int_S   f(X,Y,z) \, dz - \int_D   f(X,Y,z) \, dz \Bigg | \leq Vol(S \Delta D) < 16 \pi \epsilon + 64 \epsilon^2$$

To control the second term, note $|1-|X||, |1-|Y||  < 8 \epsilon$, and that $\int_D   f(X,Y,z) |X| |Y| \, dz \leq \pi^2$. This yields the following estimate:
\[ \Bigg | \int_D   f(X,Y,z) |X| |Y| \, dz - \int_D   f\left( X,Y,z \right) \, dz \Bigg | < ((1+8 \epsilon)^2-1) \pi^2 = \pi^2 \epsilon+64 \pi^2 \epsilon^2 \]

Therefore, the first two term are $O(\epsilon)$, and so $O(\epsilon^{1/2})$ as well.
To show that the final term is $O(\epsilon^{1/2})$, we show that that  $\int_D   f(X, Y,z) \, dz$ is H\"older $1/2$ continuous in $X$ and $Y$, when $X$ and $Y$ are near the boundary of the disk. Referring back to Figure 1, $ \int_D   f(X, Y,z) \, dz$ is the area between the two parallel lines outside of the smaller disk. 
Using Euclidean geometry, we can write $\int_D   f(X, Y,z) \, dz$ explicitly. 
In most cases, this is given by the following formula.

\begin{eqnarray*}
\int_D   f(X, Y,z) \, dz  = &  \frac{1}{2}( \arccos(\frac{X \cdot(X-Y)}{|X-Y|}) - \arccos(\frac{Y \cdot(X-Y)}{|X-Y|}) \\
& - \sqrt{1 - (\frac{X \cdot(X-Y)}{|X-Y|})^2 } +  \sqrt{1-(\frac{X \cdot(X-Y)}{|X-Y|})^2 } )\\
& - r^2 \arccos(\frac{d^2+r^2 -1}{2dr}) + \arccos(\frac{d^2+1 -r^2}{2d})\\
& - \frac{1}{2} \sqrt{(-d+r+1)(-d+r-1)(d-r+1)(d+r+1)}
\end{eqnarray*}

Here, $d = |X+Y|/2$ and $r = |X-Y|/2$. For a given pair $X, Y$, it is possible that there are fewer terms. For instance, if the small disk is contained entirely in the larger one or one of the lines does not intersect the unit disk, there will be fewer terms. By inspection, this is H\"older-1/2 continuous whenever $|X-Y|$ and $|X+Y|$ are not too small (say both greater than $48 \epsilon$).  

We can use geometry to show that this expression is small whenever $|X-Y| < 64 \epsilon$ or $|X+Y|< 64 \epsilon$. To do so, note that that $\int_D   f(X, Y,z) \, dz < 2 |X-Y|$, as the disk has diameter 2 and the width of the parallel lines in Figure 1 is $|X-Y|$. Therefore, when $|X-Y|$ is small, $\int_D   f(X, Y,z) \, dz$ is Lipschitz in $|X-Y|$.
Furthermore, if $|X+Y|$ is small and $X$ and $Y$ lie on the boundary of $S$ (which is at most $8 \epsilon$ from the boundary of $D$), then $X$ and $Y$ are nearly antipodes of the circle. In this case, very little of the unit disk will lie outside the disk defined by $X$ and $Y$, as shown in Figure 3.

\begin{center}
\includegraphics[width=150mm,scale=0.4]{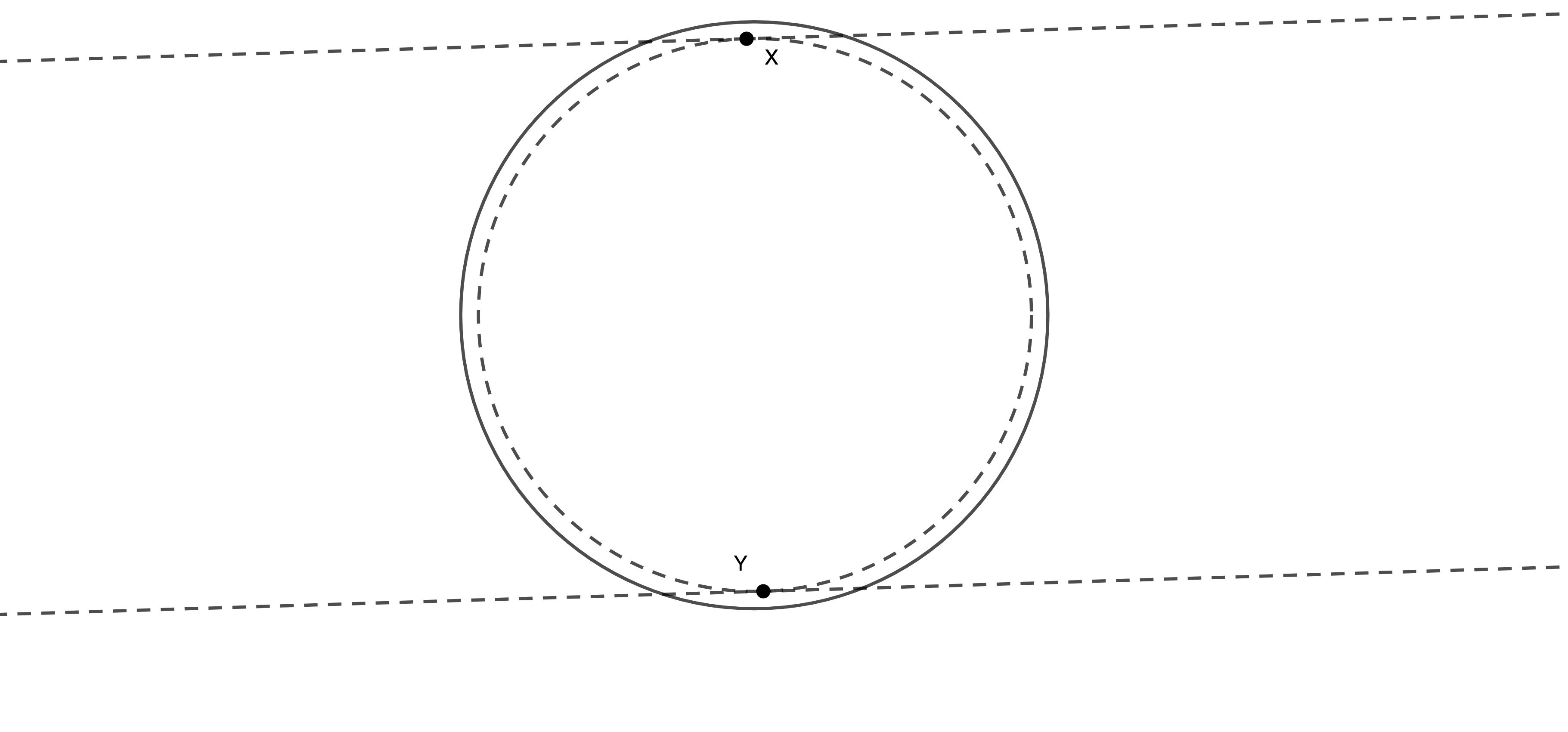}
Figure 3. 
\end{center}

To state this precisely, whenever $|X+Y|< 64 \epsilon$: 
\[ \int_D   f(X, Y,z) \, dz < \pi |1- |X|^2| +  |1- |Y|^2| < 32\pi \epsilon + 128 \pi \epsilon^2\]

\end{proof}


\subsection{Holder continuity of the terms involving $\frac{\partial f}{\partial r}$ }

In this section, we prove Lemma 15, which we restate here.

\begin{lemma*}
There exist uniform constants $ e, C >0$ so that whenever $d_{H.}(S,D) \leq \epsilon < e$, then 

\begin{eqnarray*} \Bigg | \int_0^{2 \pi} \int_S \int_S  \frac{\partial}{\partial r} f((r(\theta_1) e^{i \theta_1},y,z) r(\theta_1) (\mu(\theta_1))^2 \, dy \,d z \, d \theta & \\
-    \int_0^{2 \pi} \int_D \int_D  \frac{\partial}{\partial r} f(e^{i \theta_1},y,z) (\mu(\theta_1))^2 \, dy \,d z \, d \theta_1 \Bigg |  & < C \epsilon^{1/2} 
\end{eqnarray*}

\end{lemma*}

\begin{proof}

Before, we were content to interpret $\frac{\partial f}{\partial r}$ distributionally, without worrying about the formal definition. However, for this step we write this out explicitly using classical integrals. To make the calculation more intuitive, we write $X=r(\theta_1) e^{i \theta_1}.$

\begin{flalign*}
 \int_0^{2 \pi} \int_S \int_S  \frac{\partial}{\partial r} f(X,y,z) (\mu(\theta_1))^2 & \, dy \,d z \, d \theta_1 \\
 &  =  \int_0^{2 \pi} \int_S \int_{ \{ Z \in S | \measuredangle(YXZ)= \pi/2\} } \frac{(X-Y)}{|X-Y|} \cdot \frac{X}{|X|} |X| dZ  \,d Y \, \mu(X)^2 d \theta_1 \\
   &-   \int_0^{2 \pi} \int_S \int_{ \{ Z \in S | \measuredangle(XYZ)= \pi/2 \} } \frac{(X-Y)}{|X-Y|} \cdot \frac{X}{|X|} |X| dZ  \,d Y \, \mu(X)^2 d \theta_1 \\
   &+   \int_0^{2 \pi} \int_S \int_{ \{Z \in S | \measuredangle(XZY)= \pi/2 \}} \frac{ Z -(X+Y)/2}{|Z -(X+Y)/2|} \cdot \frac{X}{|X|} |X| dZ  \,d Y \, \mu(X)^2 d \theta_1 
\end{flalign*}

We want to show that each of these three terms are H\"older continuous in Hausdorff distance. To do this, we use we use the triangle inequality on the inner two integrals as in the previous step. For conciseness, we will only write this out fully for the first term and we denote $\bar X = \frac{X}{|X|}$.

\begin{eqnarray}
\nonumber & \Bigg | \int_S \int_{ \{ Z \in S | \measuredangle(YXZ)= \pi/2\} } \frac{(X-Y)}{|X-Y|} \cdot X dZ  \,d Y   \\ 
\nonumber &- \int_D \int_{ \{ Z \in D | \measuredangle(YXZ)= \pi/2\} } \frac{(X-Y)}{|X-Y|} \cdot \frac{X}{|X|} dZ  \,d Y  \Bigg | \\
\leq & \Bigg | \int_S \int_{ \{ Z \in S | \measuredangle(YXZ)= \pi/2\} } \frac{(X-Y)}{|X-Y|} \cdot X dZ  \,d Y \\ 
\nonumber&-\int_S \int_{ \{ Z \in D | \measuredangle(YXZ)= \pi/2\} } \frac{(X-Y)}{|X-Y|} \cdot X dZ  \,d Y \Bigg | \hfill  \\
+ & \Bigg | \int_S \int_{ \{ Z \in D | \measuredangle(YXZ)= \pi/2\} } \frac{(X-Y)}{|X-Y|} \cdot X dZ  \,d Y \\ 
\nonumber&-\int_S \int_{ \{ Z \in D | \measuredangle(Y \bar XZ)= \pi/2\} } \frac{(\bar X-Y)}{|\bar X-Y|} \cdot \bar X dZ  \,d Y \Bigg |  \\
+& \Bigg | \int_S \int_{ \{ Z \in D | \measuredangle(Y \bar X Z)= \pi/2\} } \frac{( \bar X-Y)}{| \bar X-Y|} \cdot \bar X dZ  \,d Y \\ 
\nonumber&-\int_D \int_{ \{ Z \in D | \measuredangle(Y \bar X Z)= \pi/2\} } \frac{(\bar X-Y)}{| \bar X-Y|} \cdot \bar X dZ  \,d Y \Bigg | 
\end{eqnarray}

\subsection{Proving that $(11)$ is $O(\epsilon^{1/2})$}

The title of this subsection is a bit of a misnomer. Not only must we bound (13), we need the same estimate on  
\[ \Bigg | \int_S \int_{ \{ Z \in D | \measuredangle(XYZ)= \pi/2\} } \frac{(X-Y)}{|X-Y|} \cdot X dZ  \,d Y -\int_S \int_{ \{ Z \in D | \measuredangle( \bar X YZ)= \pi/2\} } \frac{(\bar X-Y)}{|\bar X-Y|} \cdot \bar X dZ  \,d Y \Bigg | \]
\centerline{ and }
\[   \Bigg | \int_S \int_{ \{ Z \in D | \measuredangle(XZY)= \pi/2\} }  \frac{ Z -(X+Y)/2}{|Z -(X+Y)/2|} \cdot X dZ  \,d Y -\int_S \int_{ \{ Z \in D | \measuredangle( \bar XZY)= \pi/2\} }  \frac{ Z -(\bar X+Y)/2}{|Z -(\bar X+Y)/2|} \cdot \bar X dZ  \,d Y \Bigg | . \]

Since $X$ lies on the boundary of $S$, $|X-\bar X| < 8 \epsilon$, it is sufficient to show that (13) is uniformly Holder continuous in terms of $X$. To do this, we consider the innermost integral (e.g. $ \int_{ \{ Z \in D | \measuredangle(YXZ)= \pi/2\} } \frac{(X-Y)}{|X-Y|} \cdot X \,dZ$), and show the following three estimates:

\begin{enumerate}
\item Away from a small set of $Y$, this term is uniformly H\"older-$1/2$ continuous as a function of $X$.
\item The small set of $Y$'s where the uniform H\"older estimate fails (in terms of $X$) has size $O(\epsilon^{1/2})$.
\item The innermost integrals are uniformly bounded.
\end{enumerate}

 With this overview out of the way, we now do this precisely.

\begin{enumerate}

\item 
We first show that the following terms are H\"older continuous in $X$, so long as $Y$ avoids a small set:
 $$\int_{ \{ Z \in D | \measuredangle(YXZ)= \pi/2\} } \frac{(X-Y)}{|X-Y|} \cdot X dz $$ 
$$\int_{ \{ Z \in D | \measuredangle(XYZ)= \pi/2 \} } \frac{(X-Y)}{|X-Y|} \cdot X  dz$$
$$\int_{ \{Z \in D | \measuredangle(XZY)= \pi/2 \}} \frac{ Z -(X+Y)/2}{|Z -(X+Y)/2|} \cdot X dz$$


To do this, we show that both the integrand and bounds of the integral are uniformly H\"older continuous in $X$. 

For the first two terms, the integrand is H\"older-continuous in $X$, so long as $|X-Y|$ is not too small. 
Furthermore, if $|(X-Y)\cdot X| > \epsilon^{1/2}$ and $|X-Y| > \epsilon^{1/2}$, then the bounds of the first two integrals are H\"older continuous in $X$. To see this, note that the length of a chord through a point is uniformly Lipschitz in the angle that defines the chord, unless the point lies near the boundary on the disk and chord is nearly tangent to the boundary. Furthermore, the angle defining the chord is H\"older $1/2$ continuous in $X$ whenever $|X-Y| > \epsilon^{1/2}$. 

The integrand in the third term is uniformly H\"older continuous, except on the set where $|Z -(X+Y)/2|$ is small. Since $Z$ lies on a semicircle with diameter $\overline{XY}$, this can only happen when $|X-Y|$ is small. It is worth noting that when $Y$ is close to the line $y=0$, as the length of $\{Z \in D | \measuredangle(XZY)= \pi/2 \} $ is only H\"older continuous in $X$, not Lipschitz.

Combining all of this, we have uniform H\"older continuity of $ \int_D  \frac{\partial}{\partial r} f(X,Y,z) dz$ for $X$ near the boundary of $D$, so long as $Y$ is away from where $|(X-Y)\cdot X| \leq \epsilon^{1/2}$ and $|X-Y| \leq \epsilon^{1/2}$.

\item We now show that the measure of the set where the H\"older estimate fails is $O(\epsilon^{1/2})$.
Note that the set $\{ Y |  |X-Y| < \epsilon^{1/2} \}$ has measure at most $\pi \epsilon$. The set $\{ Y \in D | |(X-Y)\cdot X| < \epsilon^{1/2} \}$ is contained in a circular segment near $X$. We can find the area of this segment using Euclidean geometry, which shows that it is $O(\epsilon^{3/4})$.




\item Given $X, Y \in \mathbb{R}^2$ with $|X| < 1+8 \epsilon$, we can make the following three uniform estimates:
\begin{eqnarray*}
\Bigg | \int_{ \{ Z \in D | \measuredangle(YXZ)= \pi/2\} } \frac{(X-Y)}{|X-Y|} \cdot X dz \Bigg | \leq 2 |X|\leq 2(1+8 \epsilon)
 \end{eqnarray*}
 \begin{eqnarray*}
\Bigg | \int_{ \{ Z \in D | \measuredangle(XYZ)= \pi/2 \} } \frac{(X-Y)}{|X-Y|} \cdot X  dz \Bigg |  \leq 2 |X|\leq 2(1+8 \epsilon)
 \end{eqnarray*}
  \begin{eqnarray*}
\Bigg | \int_{ \{Z \in D | \measuredangle(XZY)= \pi/2 \}} \frac{ Z -(X+Y)/2}{|Z -(X+Y)/2|} \cdot X dz \Bigg | \leq 2 \pi |X| \leq 2 \pi (1+8 \epsilon)
 \end{eqnarray*}


  \end{enumerate}


   The preceding estimates imply that $  \int_D \int_D  \frac{\partial}{\partial r} f(X,y,z) d y  \,d z$ is uniformly H\"older continuous in $X$ when $X$ is near the boundary of the disk. To see this, observe that if we move $X$ by no more than $8 \epsilon$, the inner integral is H\"older continuous, except for a set of size $O(\epsilon^{1/2})$. Since the inner integral is uniformly bounded, the integral on the bad set is also $O(\epsilon^{1/2})$. As such, the change in the total integral is $O(\epsilon^{1/2})$. 
   
   \subsection{Proving that $(10)$ is $O(\epsilon^{1/2})$}
   
We now consider $(10)$ and its corresponding terms. We show the following three estimates:
 
 \begin{enumerate}  
 \item \begin{eqnarray*}
  \Bigg | \int_S \int_{ \{ Z \in S | \measuredangle(YXZ)= \pi/2\} } \frac{(X-Y)}{|X-Y|} \cdot X dZ  \,d Y& \\ 
-\int_S \int_{ \{ Z \in D | \measuredangle(YXZ)= \pi/2\} } \frac{(X-Y)}{|X-Y|} \cdot X dZ  \,d Y \Bigg |  &< C \epsilon^{1/2}
  \end{eqnarray*}
 \item \begin{eqnarray*}
  \Bigg | \int_S \int_{ \{ Z \in S | \measuredangle(XYZ)= \pi/2\} } \frac{(X-Y)}{|X-Y|} \cdot X dZ  \,d Y& \\ 
-\int_S \int_{ \{ Z \in D | \measuredangle(XYZ)= \pi/2\} } \frac{(X-Y)}{|X-Y|} \cdot X dZ  \,d Y \Bigg |  &< C \epsilon^{1/2}
  \end{eqnarray*}
  \item  \begin{eqnarray*}
 \Bigg | \int_S \int_{ \{ Z \in S | \measuredangle(XZY)= \pi/2\} }  \frac{ Z -(X+Y)/2}{|Z -(X+Y)/2|} \cdot X  dZ  \,d Y& \\ 
-\int_S \int_{ \{ Z \in D | \measuredangle(XZY)= \pi/2\} }  \frac{ Z -(X+Y)/2}{|Z -(X+Y)/2|} \cdot X dZ  \,d Y \Bigg |  &< C \epsilon^{1/2}
  \end{eqnarray*}
   \end{enumerate}

To do this, we use the triangle inequality again.

    \begin{eqnarray*}
 & \Bigg | \int_S \int_{ \{ Z \in S | \measuredangle(YXZ)= \pi/2\} } \frac{(X-Y)}{|X-Y|} \cdot X dZ  \,d Y \\ 
&-\int_S \int_{ \{ Z \in D | \measuredangle(YXZ)= \pi/2\} } \frac{(X-Y)}{|X-Y|} \cdot X dZ  \,d Y \Bigg |  \\
\leq & \Bigg | \int_S \int_{ \{ Z \in S \backslash D | \measuredangle(YXZ)= \pi/2\} } \frac{(X-Y)}{|X-Y|} \cdot X dZ  \,d Y \Bigg | \\ 
+& \Bigg | \int_S \int_{ \{ Z \in D \backslash S | \measuredangle(YXZ)= \pi/2\} } \frac{(X-Y)}{|X-Y|} \cdot X dZ  \,d Y \Bigg |
  \end{eqnarray*}
   
We now want to show that both of these terms are $O(\epsilon^{1/2})$. To do so, we use the same procedure as before. We first establish a H\"older estimate in terms of $X$ away from a small set of $Y$ values and then bound the size of the small set. Finally, we obtain a uniform estimate on the innermost integral. For the H\"older continuity away from a small set, we use the following transversality estimate. For conciseness, we omit the proof, which is a straightforward estimate using Euclidean geometry.

\begin{lemma}[Transversality estimate]
Suppose $\gamma_1$ and $\gamma_2$ are circles (or lines) which intersect transversally at an angle $\phi$. Suppose further that the curvatures of both $\gamma_1$ and $\gamma_2$ are less than some $K$. Then, given $\epsilon$ satisfying $\epsilon < \frac{1}{4K}$ and $\epsilon < \frac{\sin(\phi)}{4}$, then the length of the set $\{ x \in \gamma_1 | d(x, \gamma_2)< \epsilon \} < 8 \frac{\epsilon}{\sin (\phi)}$.
\end{lemma}


This lemma is useful because it allows us to bound the lengths of the loci of right triangles in $S \Delta D$. For instance, we can bound the length of the set $\{Z \in  S \Delta D | \measuredangle(XZY)= \pi/2 \}$ by the length of the set $\{Z | \measuredangle(XZY)= \pi/2 \textrm{ and } d(Z, D) < 8 \epsilon \}$. 

Explicitly, this lemma shows that if $\{Z | \measuredangle(XZY)= \pi/2 \} $ intersects $\partial D$ at an angle $\phi$, then so long as the inverse curvatures of $D$ and $\{Z \in D | \measuredangle(XZY)= \pi/2 \} $ and $\sin(\phi)$ are much larger than $\epsilon$, then we have the following estimate: $$ l(\{Z | \measuredangle(XZY)= \pi/2 \textrm{ and } d(Z, D) < 8 \epsilon \} ) < 64 \frac{\epsilon}{\sin (\phi)}$$
This implies that $l(\{Z \in  S \Delta D | \measuredangle(XZY)= \pi/2 \}) < 64 \frac{\epsilon}{\sin (\phi)}$.

Similarly, if $\{Z | \measuredangle(XYZ)= \pi/2 \} $ intersects $\partial D$ at an angle $\phi$, then so long as $\sin(\phi)$ is much larger than $\epsilon$, then we have the same estimate, and so $l(\{Z \in  S \Delta D | \measuredangle(XYZ)= \pi/2 \}) < 64 \frac{\epsilon}{\sin (\phi)}$. The argument for the $\{Z \in S \Delta D | \measuredangle(ZXY)= \pi/2 \} $ is then exactly the same.

 So long as the intersection angles are not too small and the curvature of $\{Z \in D | \measuredangle(XZY)= \pi/2 \} $ is not too large, this estimate shows that deforming from $D$ to $S$ does not change the integral  $\int_{S}  \frac{\partial}{\partial r} f(X,Y,z) dz$ much.
 A good strategy for this step is to consider curves which intersect at an angle $\phi> \epsilon^{1/2}$ as sufficiently transverse. This will give strong enough estimates to obtain uniform H\"older continuity. Meanwhile, the compliment of this set has size $O(\epsilon^{1/2})$.



This transversality estimate will encounter problems in three cases. 

\begin{enumerate}
\item The first issue occurs when $Y$ is close to the line $y=0$, in which case the set $\{Z \in S | \measuredangle(XZY)= \pi/2 \} $ is nearly tangent to $\partial D$ at their intersection, so the length of $\{Z \in S | \measuredangle(XZY)= \pi/2 \} $ might differ greatly from $\{Z \in D | \measuredangle(XZY)= \pi/2 \} $. We discard the set of $Y$'s with $y$ coordinate smaller that $\epsilon^{1/2}$, which has measure at most 2$\epsilon^{1/2}.$ Outside of this set, the length of the circle in $S \Delta D$ is $O(\epsilon^{1/2}).$
\item Another problem occurs in a small segment containing the point $-X$, in which case the line $\{ Z \in D | \measuredangle(XYZ)= \pi/2 \}$ is nearly tangent to the disk at $-X$. The strategy to circumvent this is exactly the same as in the previous step, where we discarded the set of $Y$'s so that $|(X-Y)\cdot X| < \epsilon^{1/2}$.
\item The final problem occurs when $|X-Y|$ is small, as the curvature of $\{Z \in D | \measuredangle(XZY)= \pi/2 \} $ blows up. However, the set where $|X-Y|<\epsilon^{1/2}$ has size $O(\epsilon)$, and the set $\{Z \in S | \measuredangle(XZY)= \pi/2 \} $ has length less than $2 \pi \epsilon^{1/2}$. Outside of this set, we can use the transversality estimate to obtain H\"older continuity. 
\end{enumerate}

To finish the proof, we establish the following uniform estimate.
\begin{flalign*} \Bigg | \int_{ \{ Z \in S | \measuredangle(YXZ)= \pi/2\} } \frac{(X-Y)}{|X-Y|} \cdot X dz \Bigg |  \\
   + \Bigg |  \int_{ \{ Z \in S | \measuredangle(XYZ)= \pi/2 \} } \frac{(X-Y)}{|X-Y|} \cdot Xdz  \Bigg | \\
   +  \Bigg |  \int_{ \{Z \in S | \measuredangle(XZY)= \pi/2 \}} \frac{ Z -(X+Y)/2}{|Z -(X+Y)/2|} \cdot X dz \Bigg |  \\
   \leq 2 \pi (1+ 8 \epsilon) + 4 (1+ 8 \epsilon) 
\end{flalign*}

   \subsection{Proving that $(12)$ is $O(\epsilon^{1/2})$}

In fact, we can prove a stronger estimate and show $(12)$ is $O(\epsilon)$. To do so, we must bound the following terms:

$$\int_{S \Delta D} \int_{ \{Z \in D | \measuredangle(\bar XZY)= \pi/2 \}}  \frac{ Z -(\bar X+Y)/2}{|Z -(\bar X+Y)/2|} \cdot \bar X dZ \, dY $$
$$\int_{S \Delta D} \int_{ \{ Z \in D | \measuredangle(Y \bar X Z)= \pi/2\} } \frac{(\bar X-Y)}{|\bar X-Y|} \cdot \bar X dZ  \,d Y $$
$$\int_{S \Delta D} \int_{ \{ Z \in D | \measuredangle(\bar X Y Z)= \pi/2\} } \frac{(\bar X-Y)}{|\bar X-Y|} \cdot \bar X dZ  \,d Y. $$

This term is quite a bit easier to estimate than the previous ones. For a fixed $\bar X \in \partial D$, for any $Y \in S \Delta D$, the length of the set of right triangles in $D$ with vertices $X$ and $Y$ is at most $4+2 \pi$, (twice the diameter plus the circumference). Therefore, the inner integrals have a uniform estimate. Since $S \Delta D$ is contained entirely in the annulus of outer radius $1+8 \epsilon$ and inner radius $1-8 \epsilon$, $Vol(S \Delta D) \leq 32 \pi \epsilon$, and hence $(12)$ is bounded by  $32 \pi (4+2 \pi) \epsilon$.

\subsection{Finishing the proof of Lemma 15}

From the previous three estimates in this lemma, we have shown that $ \frac{\partial}{\partial r} \int_S \int_S  f(X,Y,Z) |X| d Y  \,d Z$ is H\"older continuous in $X$ and $S$. 

Integrating with respect to $\theta$ from $0$ to $2 \pi$, this implies that
$$\int_0^{2 \pi} \int_S \int_S  \frac{\partial}{\partial r} f(r(\theta_1) e^{i \theta_1},y,z) r(\theta_1) (\mu(\theta_1))^2 \, dy \,d z \, d \theta_1$$
 is H\"older continuous in $S$, and so the proof of the lemma is complete.

\end{proof}

\subsection{Lipschitz continuity of the arc-length term}

In this section, we prove Lemma 16, which we restate here. We refer to this as the arc-length term because it corresponds to the change in arc-length as the region evolves along the canonical homotopy.

\begin{lemma*} There exist uniform constants $ e, C >0$ so that whenever $\epsilon < e$ and $d_{H.}(S,D) \leq \epsilon$, 
\begin{eqnarray*}
 \Bigg | \int_0^{2 \pi} \int_S \int_S  f(r(\theta_1, t) e^{i \theta_1},y,z) \mu(\theta_1)^2 \, dy \,d z \, d \theta_1 -
  \int_0^{2 \pi} \int_D \int_D  f(e^{i \theta_1},y,z) \mu(\theta_1)^2 \, dy \,d z \, d \theta_1 \Bigg | < C \epsilon
  \end{eqnarray*}
\end{lemma*}

\begin{proof}
To do this, we use the triangle inequality:

\begin{eqnarray*}
 \Bigg | \int_0^{2 \pi} \int_S \int_S  f(r(\theta_1, t) e^{i \theta_1},y,z) \mu(\theta_1)^2 \, dy \,d z \, d \theta_1 -  \int_0^{2 \pi} \int_D \int_D  f(e^{i \theta_1},y,z) \mu(\theta_1)^2 \, dy \,d z \, d \theta_1 \Bigg | \\
   \leq  \Bigg | \int_0^{2 \pi} \int_S \int_S  f(r(\theta_1, t) e^{i \theta_1},y,z) \mu(\theta_1)^2 \, dy \,d z \, d \theta_1 -  \int_0^{2 \pi} \int_D \int_D  f(r(\theta_1, t) e^{i \theta_1},y,z) \mu(\theta_1)^2 \, dy \,d z \, d \theta_1 \Bigg | \\
  +  \, \Bigg | \int_0^{2 \pi} \int_D \int_D  f(r(\theta_1, t) e^{i \theta_1},y,z) \mu(\theta_1)^2 \, dy \,d z \, d \theta_1 -  \int_0^{2 \pi} \int_D \int_D  f(e^{i \theta_1},y,z) \mu(\theta_1)^2 \, dy \,d z \, d \theta_1 \Bigg |
  \end{eqnarray*}
  
We can estimate the first term by estimating it as a single integral over $(S \times S) \Delta (D \times D)$. Using the same estimate as in the Lipschitz estimate of $p(S)$, the volume of $(S \times S) \Delta (D \times D)$ is less than $16 \pi (1+8\epsilon)^3 \epsilon$. Furthermore, $|f(X,Y,Z)| \leq 1$, which shows the following estimate:
   $$|\int_S \int_S  f(r(\theta_1, t) e^{i \theta_1},y,z) \, dy \,d z -\int_D \int_D  f(r(\theta_1, t) e^{i \theta_1},y,z) \, dy \,d z| < C \epsilon$$
  
To estimate the second term, we consider the following: 

$$\int_D \int_D  f(r(\theta_1, t) e^{i \theta_1},y,z) \, dy \,d z $$

Note that this is the total mass of acute triangles with two vertices in a disk given a fixed third vertex (which may or may not be in the disk). We want to show that this mass is uniformly Lipschitz in the choice of fixed vertex. To see this, we can estimate $|\dt \int_D \int_D  f(\bar X+tV,y,z) \, dy \,d z|$ for any point $ \bar X$ and unit vector $V$.

We can use the triangle inequality and some estimates to bound this quantity.
\begin{eqnarray*}
\Bigg | \dt \int_D \int_D  f(\bar X+tV,y,z) \, dy \,d z \Bigg | & \leq& \Bigg | \int_D \int_{ \{Z \in D | \measuredangle(\bar XZY)= \pi/2 \}}  \frac{ Z -(\bar X+Y)/2}{|Z -(\bar X+Y)/2|} \cdot V dZ \, dY \Bigg | \\
& + & \Bigg | \int_D \int_{ \{ Z \in D | \measuredangle(Y \bar X Z)= \pi/2\} } \frac{(\bar X-Y)}{|\bar X-Y|} \cdot V dZ  \,d Y \Bigg | \\
& + & \Bigg | \int_D \int_{ \{ Z \in D | \measuredangle(\bar X Y Z)= \pi/2\} } \frac{(\bar X-Y)}{|\bar X-Y|} \cdot VdZ  \,d Y \Bigg | \\
& \leq& \Bigg | \int_D \int_{ \{Z \in D | \measuredangle(\bar XZY)= \pi/2 \}} 1~ dZ \, dY \Bigg | \\
& + & \Bigg | \int_D \int_{ \{ Z \in D | \measuredangle(Y \bar X Z)= \pi/2\} } 1~ dZ  \,d Y \Bigg | \\
& + & \Bigg | \int_D \int_{ \{ Z \in D | \measuredangle(\bar X Y Z)= \pi/2\} } 1~ dZ  \,d Y \Bigg |
  \end{eqnarray*}
 We can geometrically bound each of these.
\begin{eqnarray*}
\Bigg | \dt \int_D \int_D  f(X+tV,y,z) \, dy \,d z| & \leq& | \int_D 2 \pi \, dY |+ | \int_D 2  \,d Y|  +  |\int_D 2 \,d Y \Bigg | \\
& =& \pi (2\pi+4) \\
  \end{eqnarray*}

This shows that the inner two integrals are uniformly Lipschitz. Integrating with respect to $\theta_1$, this implies that for $\epsilon$ small, we can find a uniform $C$ so that
\[ \Bigg | \int_0^{2 \pi} \int_D \int_D  f(r(\theta_1, t) e^{i \theta_1},y,z) \mu(\theta_1)^2 \, dy \,d z \, d \theta_1 -  \int_0^{2 \pi} \int_D \int_D  f(e^{i \theta_1},y,z) \mu(\theta_1)^2 \, dy \,d z \, d \theta_1 \Bigg | < C \epsilon. \]

This completes the proof of Lemma 16.

\end{proof}

Since we have proven Lemmas 14-16, we have established the H\"older-$1/2$ continuity for each of the terms in $\frac{d^2}{dt^2} M$. This completes the proof of Theorem 3.

\end{proof}

\end{document}